\documentclass[a4paper,12pt]{article}
\usepackage[english]{babel}
\usepackage[tbtags]{amsmath}
\usepackage{amssymb}
\usepackage{amsfonts}
\usepackage{amsthm}
\usepackage{calrsfs}
\usepackage{array}
 \usepackage{bbm}
\usepackage{stmaryrd}
\usepackage{upgreek}
\usepackage[svgnames]{xcolor}
\usepackage{hyperref}
\usepackage{mathtools}
\newenvironment{pf}{\vspace{0.5em}\noindent\textbf{Proof.} }{\quad \hfill $\Box$ \\ \vspace{0.5em}\\}

\def\div{\, \mbox{div}\,  }

\def\L1{L^r(\er^N)}
\def\L2{H^1_{\rho}(\er^N)}

\def\vvv { \varepsilon_1}

\def\vv { \varepsilon}
\def\pp{ p-\varepsilon}
\def\w{\rho (y)}
\def\p{\phi }
\def\ps{\phi(s)}
\numberwithin{equation}{section}
\def\ibint{\int_{\er^N}}
\def\iint {\int_{\er^N}}

\def\ia{\int_{s}^{s+1}}

\def\L{L(w(s),s)}

\def\v{{\mathrm{d}}v}

\def\Box{\hfill\rule{2.5mm}{2.5mm}}
\def\t{{\mathrm{d}}\tau}

\def\y{{\mathrm{d}}y}

\def\no{\nonumber}

\def \T {\max(-\log T,1)}
\def \TT {\max(-\log T,S_1)}
\def \TTT {\max(-\log T,S_2)}

\def\grad{\nabla}

\def\m {{m}}

\def \er {\mathbb R}

\def\R{{\mathbb {R}}}

\def\cprime{$'$}

\newcommand{\ds}{\displaystyle}

\def\J{\frac{1}{s}}

\def\t{{\mathrm{d}}\tau}

\def\y{{\mathrm{d}}y}

\def\y{{\mathrm{d}}y}

\theoremstyle{plain}
\newtheorem{thm}{Theorem}
\newtheorem*{thm*}{Theorem}
\newtheorem{coro}[thm]{Corollary}

\newtheorem{prop}{Proposition}[section]

\newtheorem{lem}[prop]{Lemma}

\theoremstyle{definition}

\theoremstyle{remark}
\newtheorem{nb}{Remark}[section]

\makeatletter
\def\blfootnote{\xdef\@thefnmark{}\@footnotetext}
\makeatother

\title{\bf The blow-up rate for    a non-scaling invariant   semilinear heat equation }

\author{Mohamed  Ali  Hamza\\
{\it \small % Department of Basic Sciences Deanship of Preparatory and Supporting Studies},\\
Imam Abdulrahman Bin Faisal University
P.O. Box 1982 Dammam, Saudi Arabia}\\
Hatem Zaag\\
{\it \small Universit\'e Sorbonne Paris Nord,}\\
{\it \small LAGA, CNRS (UMR 7539), F-93430, Villetaneuse, France}
}

\date{}

\allowdisplaybreaks

\begin{document}

\maketitle

\begin{abstract}
 We consider  the    semilinear heat
equation $$\partial_t u -\Delta u =f(u), \quad  (x,t)\in \er^N\times [0,T),\qquad (1)$$ with
$f(u)=|u|^{p-1}u\log^a (2+u^2)$,  where $p>1$ 
is Sobolev subcritical 
and $a\in \er$.
We first  show an upper bound for any 
 blow-up solution  of (1).  Then,  using this estimate and  the logarithmic property, we prove  that   the exact blow-up rate of any singular solution of (1)  is  given by the ODE solution  associated with $(1)$,
namely $u' =|u|^{p-1}u\log^a (2+u^2)$. In other terms, all blow-up solutions in the Sobolev subcritical range are Type I  solutions. Up to our knowledge, this is the first 
determination of the blow-up rate for a semilinear heat equation where the main nonlinear term is not homogeneous.
 % Unlike the pure power  case ($g(u)=|u|^{p-1}u$) the difficulties here  are due to the fact that  equation (1) is not scale  invariant.
\end{abstract}

\medskip



%%%%%%%%%%%%%%%%%%%%%%%%%%%%%%%%%%%%%%%%%%%%%%%%%%%%%%%%%%%%%%%%%%%%%%%%%%%%%
%%%%%%%%%%%%%%%%%%%%%%%%%%%%%%%%%%%%%%%%%%%%%%%%%%%%%%%%%%%%%%%%%%%%%%%%%%%%%
\section{Introduction}
\subsection{ Motivation of the problem }
%%%%%%%%%%%%%%%%%%%%%%%%%%%%%%%%%%%%%%%%%%%%%%%%%%%%%%%%%%%%%%%%%%%%%%%%%%%%%
%%%%%%%%%%%%%%%%%%%%%%%%%%%%%%%%%%%%%%%%%%%%%%%%%%%%%%%%%%%%%%%%%%%%%%%%%%%%%
This paper is devoted to the study of blow-up solutions for the
following semilinear  heat equation:
\begin{equation}\label{gen}
\left\{
\begin{array}{l}
\partial_t u =\Delta u +f(u),\qquad  (x,t)\in \er^N\times [0,T),\\
\\
u(x,0)=u_0(x)\in  L^{\infty}(\er^N),
\end{array}
\right.
\end{equation}
%in spatial dimensions $N$,
where $u(t):x\in{\er^N} \rightarrow u(x,t)\in{\er}$ 
with focusing nonlinearity $f$ defined by:
\begin{equation}\label{deff}
 f(u)=   |u|^{p-1}u\log^a (2+u^2), \quad  p>1,\quad  a\in \er.
 \end{equation}
%The spaces  $L^{2}_{loc,u}(\er^N)$ and $H^{1}_{loc,u}(\er^N)$ are  defined by
%\begin{equation*}
%L^{2}_{loc,u}(\er^N)=\{u:\er^N\rightarrow \er/ \sup_{d\in \er^N}(\int_{|x-d|\le 1}|u(x)|^2dx)<+\infty \},
%\end{equation*}
%and
%\begin{equation*}
%H^{1}_{loc,u}(\er^N)=\{u\in L^{2}_{loc,u}(\er^N),|\grad u|\in L^{2}_{loc,u}(\er^N) \}.
%\end{equation*}
We assume in addition  that $p>1$  and if $N\ge 3$,  we further  assume that
\begin{equation}\label{subc}
p<p_S\equiv \frac{N+2}{N-2}.
\end{equation}
Note that when $a\neq 0$, the nonlinear term is not homogeneous, and
this is the focus of our paper.

\medskip

By standard results the problem \eqref{gen} has  a unique solution for any   $u_0\in  L^{\infty}(\er^N)$. More precisely, there is a unique maximal solution on $[0,T)$, with $T\le \infty$.  If  $T<\infty$,
we say that  the solution of \eqref{gen}
 blows up in finite time. In that case, it holds that  $\|u(t)\|_{L^{\infty}(\er^N)}\to \infty$ as $t\to T$.
% In this  case we call the
%solution $u$ to  \eqref{gen}  blows up in finite time and $T$ the blow-up time. 
%
Such a solution $u$ is
called a blow-up solution of  \eqref{gen} with the blow-up time $T$. 

\medskip

%Before stating the main issues in this paper, we will review the known results related.

In the case $a=0$, equation \eqref{gen} reduces to the semilinear heat equation with power nonlinearity:
\begin{equation}\label{NLH}
\partial_{t} u =\Delta  u+|u|^{p-1}u,   \,\,\,(x,t)\in \R^N \times [0,T).
\end{equation}
In the literature, the determination of the blow-up rate has been linked to the terminology of  ''Type I/Type II solutions'', first introduced (up to our knowledge) by Matano and Merle in 
\cite{MMCPAM2004}.
In that paper, if a solution $u$ to \eqref{NLH} blows up at time $T$ and satisfies for all $t\in [0,T)$,
 \begin{equation}\label{mzmz}
 \|u(t)\|_{L^{\infty}(\er^N)}\le C  (T-t)^{-\frac{1}{
   p-1}},
  \end{equation}
% where $I(\bar x,\bar t)=[\bar x-\bar t,\bar x+\bar t]$ and,
for  some positive constant $C$, independent of time $t$,
then $u$ is called a Type I. If not, then $u$ is said to be of Type II. Note that the bound given in (1.5) is (up to a multiplying factor) a solution of the associated ODE $u'=u^p$. \\
In the subcritical case under consideration (1.3), we know from Giga and Kohn  \cite{GKcpam85,  GKIuMJ87,GKCPAM89},   and also Giga, Matsui and Sasayama \cite{GMSiumj04} that \textit{all} blow-up solutions of (1.4) are of Type I. Moreover, from the construction provided by Nguyen and Zaag \cite{VZ}, we know that Type I solutions are available for any superlinear exponent $p$, not only in the subcritical case, despite what the authors noted at that time. \\
As for Type II solutions, we know that they are available in the critical range (see Schweyer \cite{sJFA2012}, Harada \cite{HIHP20}, Del Pino, Musso and Wei \cite{DMW}, Collot, Merle and Raphaël \cite{CMR}, Filippas, Herrero and Vel\`azquez \cite{FHV}), and also in the supercritical range (see Herrero and Vel\`azquez \cite{HV1}, Mizoguchi \cite{M}, Seki \cite{S1,S2}.

\bigskip 

Going back to the proof given in  \cite{GMSiumj04} for the fact that all blow-up solutions for equation \eqref{NLH} in the subcritical range \eqref{subc} are of Type I, we would like to mention that the following estimate is central in the argument:
%
%It should be noted that   Giga and Kohn  in \cite{GKcpam85,  GKIuMJ87,GKCPAM89},  proved that
%the result \eqref{mzmz} 
%for   a restricted range of $p$ ($1<p<\frac{3N+8}{3N-4}$) or   for non-negative initial data). 
%Later  Giga, Matsui  and Sasayama in  \cite{GMSiumj04} extended 
%the estimate   \eqref{mzmz}  in the subcritical case  ($1<p< p_S$) 
% without assuming
%non-negativity for initial data. The extension  
% in \cite{GMSiumj04}  is obtained  thanks to the  key integral estimate:
%%
%
% Note that this result about the blow-up rate  is valid   
% in the subcritical case  ($1<p< p_S$). 
%
\begin{equation}\label{23jan1}
\int_s^{s+1} \|w(\tau )\|_{L^{p +1}(\mathbf{B}_R)}^{(p+1)q} {\mathrm{d}\tau}  \leq K(q,R), 
 \ \   \forall q \geq 2,\ \   \forall R>0, \ \   \forall s > -\log T,
\end{equation}
where $w$ is the similarity variables version  of the solution defined in \eqref{scaling}  below and 
$\mathbf{B}_R\equiv
  B(0,R)$  is the  open ball of radius $R$ centered at the origin in $\er^N.$

\medskip

Exploiting the   non-trivial perturbative method  introduced   by the authors in  \cite{HZjhde12, HZnonl12}  in the hyperbolic case
and arguing as in the  non perturbed case   in \cite{GMSiumj04},
 Nguyen proved in \cite{V}
% and by combining  the  non-trivial perturbative method  and the strategy used in the, 
  a similar result to $\eqref{mzmz}$,
valid in the subcritical  case,  for  a class of strongly perturbed
semilinear heat equations
\begin{equation}\label{NLHP}
\partial_t u =\Delta u+|u|^{p-1}u+h( u ),\,\,\,(x,t)\in \R^N \times [0,T),
 \end{equation}
under the assumptions  $|h(u)|\leq M(1+|u|^p)\log^{-a}(2+u^2)$,   for some $M > 0$ and $a>1$. 
%Let us denote that
%the existence of the Lyapunov 
%functional  for equation \eqref{NLHP}  is a crucial step.  
%
%\medskip
%
% Let us note that 
% in     \cite{V}, the author    considers  a  class of
%  perturbed equations   where the nonlinear term is  equivalent to the pure power $|u|^{p-1}u$   and  obtains  the  estimate 
%  \eqref{mzmz}.
Obtaining the same blow-up rate is reasonable, since 
%This is due to the fact that
the dynamics is still  governed by the ODE $u' =|u|^{p-1}u$.
% However, in the present article the situation is different.
Furthermore, the proof remains (non trivially) perturbative  with respect to the homogeneous PDE \eqref{NLH}, which is scale invariant.

\bigskip
This leaves unanswered an interesting question: is the scale invariance  property crucial  in deriving the blow-up rate?

\bigskip
 
In fact we {\textit{had the impression}} that the answer was ''yes'', since the scaling invariance induces in  similarity  variables a PDE which is  autonomous in the unperturbed case \eqref{NLH}, and asymptotically autonomous in the perturbed case \eqref{NLHP}.

\bigskip

In this paper we {\textit{prove}} that the answer is ''no'' from the example  of the non  homogeneous PDE \eqref{NLH}. In fact, 
our situation is different from  \eqref{NLH} and  \eqref{NLHP}.
 Indeed,   the  term   $|u|^{p-1}u\log^a (2+u^2)$  
 is playing a fundamental role in the dynamics of  the blow-up  solution of \eqref{gen}. More  precisely, we obtain an analogous  result to \eqref{mzmz} but with a logarithmic correction  as shown in \eqref{Heatest} below.  In fact, the bow-up rate  is
 given by the solution of the associated ODE   $u'=|u|^{p-1}u\log^a (2+u^2)$.

\medskip

In this paper, we study the  blow-up rate of any singular solution of \eqref{gen}.
% Before going on, it is necessary
%to mention  that 
% the blow-up rate in the case
%with pure  power nonlinearity 
Before handling the PDE, we first consider the following ODE associated to \eqref{gen}:
\begin{equation}\label{v}
v_T' (t)=  |v_T (t)|^{p-1}v_T(t) \log^{a}\big(v_T^2(t)  +2\big), \quad v(T)=\infty,
\end{equation}
and show  that  the nonlinear term including  the  logarithmic factor
gives rise to different dynamics. In fact, thanks to \cite{DVZ} (see Lemma A1), 
% this paper, we show the blow-up rate is given by \eqref{v}.
 we can see that  the solution $v_T$   satisfies
 \begin{equation}\label{equivv}
 v_T(t) \sim \kappa_{a}\psi_T(t), \text{ as } t \to\ T,\quad \textrm{ where}\quad  \kappa_{a} =  \left(\frac{2^{a}}{(p-1)^{1-a}} \right)^{\frac1{p-1}},
 \end{equation}
and
 \begin{equation}\label{psi}
 \psi_T(t)=(T - t)^{-\frac{1}{p-1}} (-\log (T - t))^{-\frac{a}{p-1}}.
 \end{equation}

\medskip

Therefore, it is natural to  extend  the terminology
 "Type I/Type II solutions" for  the blow-up of a solution $ u(x, t)$ of  \eqref{gen} by the following:
 \begin{align}\label{BlowupIII}
   (T-t)^{\frac{1}{
   p-1}}(-\log (T - t))^{\frac{a}{p-1}}\|u(t)\|_{L^{\infty}(\er^N)}\le C, \qquad & \textrm {Type I}\\
  {\lim \sup}_{t\to T}(T-t)^{\frac{1}{
   p-1}}(-\log (T - t))^{\frac{a}{p-1}}\|u(t)\|_{L^{\infty}(\er^N)}=\infty, \qquad & \textrm {Type II}.
  \end{align}
%
%There are two types of blow-ups: the blow-up of a solution u(x, t) is of type I if it happens at most at
%the ODE rate:
%lim sup
%t→T
%(T − t)
%1
%p−1 ku(·, t)kL∞(Rn) < +∞
%while the blow-up is said of type II if
%lim sup
%t→T
%(T − t)
%1
%p−1 ku(·, t)kL∞(Rn) = +∞.

\medskip

Let us mention that % the blow-up question for  the    solution of equation   \eqref{gen}
  Duong, Nguyen and Zaag construct in \cite{DVZ} a solution of equation   \eqref{gen}
which blows up in finite time $T$,  only at one blow-up point $x_0$, according to the following asymptotic dynamics:
\begin{equation}\label{heatequiv}
u(x,t)\sim v_T(t)\Big(1+\frac{(p-1)|x-x_0|^2}{4p(T-t)|\log (T-t)|}\Big)^{-\frac{1}{p-1}},\qquad  as\  t\to T,
\end{equation}
  where $v_T(t)$ is the  solution of \eqref{v}   with an equivalent given in \eqref{equivv}. Note from \eqref{heatequiv}  that the constructed solution is
%Given that we construct a blow-up solution   of \eqref{gen}
 of Type I. 
% have the same expression in the pure power nonlinearity case \eqref{NLH}   with   $v_T(t)$ replaced by 
% $\kappa (T-t)^{-\frac1{p-1}}$  (see \cite{BKnonl94}), we see that the effect of the nonlinearity is all encapsulated in the ODE %\eqref{v}. 

\bigskip

Concerning the blow-up rate  for the hyperbolic equations with a non-homogeneous main term, we would like to mention that 
 in \cite{HZjmaa2020} and \cite{HZArxiv2020},  we consider   the    semilinear wave
equation 
\begin{equation}\label{WP}
\partial_t^2 u -\Delta u =|u|^{p-1}u\log^a (2+u^2), \quad  (x,t)\in \er \times [0,T),
\end{equation}
   where $a\in \er$ and $p>1$ is subconformal, in the sense that $(N-1)p<N+3$. 
 We prove  that   the exact blow-up rate of any singular solution of \eqref{WP}  is  given by the ODE solution  associated with \eqref{WP},
namely
% $u'' =|u|^{p-1}u\log^a (2+u^2)$.
%  Unlike the pure power  case ($g(u)=|u|^{p-1}u$) the difficulties here  are due to the fact that  equation (1) is not scale  invariant.
%
%Before handling the PDE, we first study the associated ODE  to \eqref{gen}
\begin{equation}\label{vw}
V_T'' (t)=  |V_T (t)|^{p-1}V_T(t) \log^{a}\big(V_T^2(t)  +2\big), \quad V(T)=\infty.
\end{equation}
% Recently, the   result proved  in one space
% dimension has been  established in higher dimensions by authors in \cite{HZArxiv2020}.
Let us mention that    the nonlinear term involving  the  logarithmic
factor   gives raise to different dynamics. To be precise,   the solution $V_T$   satisfies
 \begin{equation}\label{equivvw}
 V_T(t) \sim C (a,p)(T - t)^{-\frac{2}{p-1}} (-\log (T - t))^{-\frac{a}{p-1}}, \text{ as } t \to\ T.\ \  
\end{equation}
Since the blow-up rate is given by $V_T(t)$,
we see that the effect of the nonlinearity is completely encapsulated in  \eqref{equivvw}. 
Note that before  \cite{HZjmaa2020, HZArxiv2020}, we could successfully implement our perturbative method in  \cite{H1, omar1, omar2,HZjhde12,HZnonl12}  
 to derive the blow-up rate
 for some classes of perturbed wave equations where the main nonlinear term is power-like (hence, homogeneous). 
\subsection{Strategy of the proof}
% Consider $u $   a solution of ({\ref{gen}}) with
%blow-up graph $\Gamma:\{x\mapsto T(x)\}$.
%Assuming $T(x_0)\le T_0(x_0,p,a,N)$ small enough (without loss of generality) 
Going back to the equation under study in this paper (see \eqref{gen} and \eqref{deff}), we
 introduce the following similarity variables,   defined for all
  $x_0\in \er^N$ by:
\begin{equation}\label{scaling}
y=\frac{x-x_0}{\sqrt{T-t}},\quad s=-\log (T-t),\quad u(x,t)=\psi_{T}(t)w_{x_0,T}(y,s),
\end{equation}
where $\psi_T(t)$ is the explicit rate given in \eqref{psi}.
On may think that it would be more natural to replace $\psi_T(t)$ by
$v_T(t)$ (defined in \eqref{v}) in this definition, since the latter is an exact solution of 
the ODE \eqref{v}. That might be good, however, as $v_T(t)$ has no explicit expression, the calculations will immediately become too complicated.
For that reason, we preferred to replace the non-explicit $v_T(t)$ by its explicit equivalent $\psi_T(t)$ in \eqref{psi}. The fact the latter is not an exact solution of \eqref{v} will have no incidence in our analysis.

\medskip

From (\ref{gen}) and  (\ref{scaling}), the  function $w_{x_0,T}$  (we write $w$ for
simplicity) satisfies the following equation for all $y\in \er^N$ and $s\ge \max ( -\log T,1)$:
\begin{align}\label{A}
\partial_{s}w&=\frac{1}{\rho}\div(\rho \grad w)-\frac{1}{p-1}(1-\frac{a}{s})w+
e^{-\frac{ps}{p-1}}s^{\frac{a}{p-1}} f(\ps w),
\end{align}
where   
%$\rho (y)=(1-|y|^2)^{\alpha}$,
\begin{equation}\label{rho}
\rho (y)=e^{-\frac{|y|^2}4}%({4\pi})^{-\frac{N}2}
%e^{-\frac14\sum_{i=1}^Ny_i^2}
\end{equation}
and 
\begin{equation}
\ps =e^{\frac{s}{p-1}}s^{-\frac{a}{p-1}}.\label{defphi}
\end{equation}

\medskip

% This change of variables  is associated to  the nonlinear heat equation including a  logarithmic nonlinearity \eqref{gen}.
%In fact, we have the same transformation as in the 
% pure power  case ($g(u)=|u|^{p-1}u$).
% except of  $\phi(t)$ instead of  $(T_0-t)^{-\frac2{p-1}}$. Note that this 
% change of variables
%transforms the backward lightcone with vortex $(x_0,T_0)$ into the
%infinite cylinder $(y,s)\in B\times [-\log T_0,+\infty)$. 
In the new
set of variables $(y,s),$  studying the behavior of $u$ as $t \rightarrow T$
is equivalent to studying  the behavior of $w$ as $s \rightarrow +\infty$.

\medskip
%
%
%Throughout this paper,
%$C$  denotes a  generic positive constant
% depending only on $p,N$  and $a,$  which may vary from line to line. Also,  we will use $K_1$, $K_2$, $K_3$... $M_1$, $M_2$, $M_3$... $Q_1$, $Q_2$, $Q_3$... 
% as  positive constants depending  on $p,N, a$  and 
%the  initial data.
%%
%% to denote a  generic positive constant
%% depending only on $p,N, a$  and initial data   which may vary from line to line. 
%Moreover, we denote 
% by $\mathbf{B}_R$  the open ball in $\er^N$ with center $0$ and radius $R$.
%  We write $f(s)\sim g(s)$ to indicate 
%$\displaystyle{\lim_{s\to \infty}\frac{f(s)}{g(s)}=1}$. 
%Furthermore, for $r\ge 1$,  we
%denote by $L^r_{\rho}(\er^N)$ the weighted  $L^r(\er^N)$
% space: %endowed with the norm
%\begin{equation}\label{defLr}
%L^r_{\rho}(\er^N)=\Big\{ u\in L^1_{loc}(\er^N)/\ \     
%\displaystyle\ibint |u(y)|^r\w \y< \infty  \Big\}
%\end{equation}
%and by $H^1_{\rho}(\er^N)$  the space of function $u\in  L^2_{\rho}(\er^N)$
% satisfying $\grad u \in   L^2_{\rho}(\er^N)$, endowed with the norm
%\begin{equation}\label{defH1}
%\|u\|_{ H^1_{\rho}(\er^N)}=\Big(\|\grad u\|^2_{ L^2_{\rho}(\er^N)}
%+\|u\|^2_{ L^2_{\rho}(\er^N)}
%\Big)^{\frac1{2}}.
%\end{equation}
%??? useful defintion??

%
%
%The spaces  $L^{2}_{loc,u}(\er^N)$ and $H^{1}_{loc,u}(\er^N)$ are  defined by
%\begin{equation*}
%L^{2}_{loc,u}(\er^N)=\{u:\er^N\rightarrow \er/ \sup_{d\in \er^N}(\int_{|x-d|\le 1}|u(x)|^2dx)<+\infty \},
%\end{equation*}
%and
%\begin{equation*}
%H^{1}_{loc,u}(\er^N)=\{u\in L^{2}_{loc,u}(\er^N),|\grad u|\in L^{2}_{loc,u}(\er^N) \}.
%\end{equation*}

\medskip

While reading Giga and Kohn \cite{GKcpam85,GKIuMJ87,GKCPAM89} dedicated to the blow-up rate of the homogeneous case  \eqref{NLH}, one sees that the existence of a 
Lyapunov functional for the similarity variables’ version (\ref{A}) with $a=0$ is central in the argument. Clearly, 
%As we  mentioned earlier,
  the invariance of equation \eqref{NLH} under  the   scaling transformation
$u \mapsto u_{\lambda}(x,t)=\lambda^{\frac1{p-1}}u(\lambda x,\lambda^2 t)$   was  crucial
in  the construction of the Lyapunov functional.
%in similarity variables (see Giga and  Kohn
%\cite{GKcpam85,GKIuMJ87,GKCPAM89}).
  The  fact that 
 equation \eqref{gen} is not invariant under the last scaling transformation  implies 
that the existence of a Lyapunov functional in similarity variables is far from being trivial (see   \cite{V,V2} in the parabolic case and 
 \cite{H1, omar1, omar2,HZjhde12,HZnonl12,HZjmaa2020} in the hyperbolic case).

\medskip

 In this paper, we %prove 
%a   polynomial  (in $s$) space-time bound  on  the similarity variables' version 
%  of the solution $u$ of  \eqref{gen}, valid  in any dimensions  in the subconformal case. 
%However, our main contribution  lays,  in one space dimension. It consists in the
 construct a 
  Lyapunov  functional in similarity variables  for the problem \eqref{A}.   Then, we prove  that   the blow-up rate of any singular solution of \eqref{gen}  is given by  the  solution of \eqref{v}.
% following ODE:
%$u' =|u|^{p-1}u\log^a (2+u^2)$. 

\bigskip

Let us explain how we derive the Lyapunov functional. As we did for the perturbed wave equation with a conformal exponent in
 \cite{H1, omar2,HZjhde12},
  we proceed in 2 steps: \\
- Step 1: we first introduce some functional (not a Lyapunov functional) for equation \eqref{A},  which is bounded by $s^\alpha$ for some $\alpha>0$, then show that $w$ enjoys also a polynomial (in $s$) bound.\\
- Step 2: then, viewing equation \eqref{A}
 as a perturbation of the
 case of a pure power nonlinearity (case where $a=0$ in \eqref{A}) by  the following  terms: 
\begin{equation}\label{perb4}
\frac{a}{(p-1)s} w \quad \textrm{ and} \quad  e^{-\frac{ps}{p-1}}s^{\frac{a}{p-1}} f(\ps w),
\end{equation} 
 we use the rough estimates on $w$ proved in the first step, in order to control the « perturbative » terms in  \eqref{A}. This way, we find a Lyapunov functional for  \eqref{A}, then use it to prove that the solution itself is bounded.\\
Specifically, in Step 1, we would like to add the following regarding the effect of the perturbation terms \eqref{perb4} and the way we handle them:
%First, we exploit  some functional  to obtain  a rough estimate on the blow-up solution; namely a polynomial (in $s$) bound on the solution  in similarity variables.  The issue  is how to handle the perturbative terms in \eqref{A}.  In fact,
%in order to control them, we  view  equation \eqref{A}  as a perturbation of the
% case of a pure power nonlinearity (case where $a=0$ in \eqref{A}) 
%  with the following  terms: 
%\begin{equation}\label{perb4}
%\frac{a}{(p-1)s} w \quad \textrm{ and} \quad  e^{-\frac{ps}{p-1}}s^{\frac{a}{p-1}} f(\ps w).
%\end{equation} 
 The first   term 
is a    lower order term  which was already  handled  in the Sobolev subcritical  perturbative case treated in
   \cite{V,V2}.  % However, the nonlinear term depend in the time $s$ because the non-scaling  invariant of the equation  under the transformation  \eqref{}.
 However, since  the nonlinear term  
$e^{-\frac{ps}{p-1}}s^{\frac{a}{p-1}} f(\ps w)$ depends on  time $s$,
  we expect the time derivatives to be delicate. 
 Thanks to the fact that   $uf(u)-(p+1)\int_0^uf(v)
{\mathrm{d}}v\sim  
\frac{2a}{p+1}|u|^{p+1}\log^{a-1}(2+u^2)$, as $u\to \infty$,
we   construct a functional    (in Section \ref{section2})
  satisfying this kind of differential inequality:
 \begin{equation}\label{jardin1}
\frac{d}{ds}h(s)\le -\frac12 \iint (\partial_{s}w)^2{\w}\y
+ \frac{C}{s}h(s)+Ce^{-s},
\end{equation} 
 which  implies  a
polynomial estimate.

\medskip

%
% Even though the rough estimate obtained  seems bad, it is very useful to
% allow us to derive  a  Lyapunov functional for equation \eqref{A}. More precisely, we use this  polynomial estimate and the structure of the
% nonlinear term  to construct    a  Lyapunov functional for equation \eqref{A}  as  a  crucial step to derive  the optimal estimate. 
%%Let us note that, unlike the hyperbolic case, when    
%the method is valid  only in one dimensional case
%and breaks down  in higher dimensional case, this method works in any dimaentional case in the parabolic case... (Explain why its ok parabolic and no hyp). 
%In addition, it's useful to note that this   Lypanov functional  is a small perturbation of the  natural energy.
%Accordingly,
%in the rest of this paper, we  consider the  one dimensional case.

\bigskip

In order
to state our main result,  we start by introducing the following %natural 
functionals:
\begin{eqnarray}
E(w(s),s)\!\!\!&=&\!\!\!\!\iint \Big(\frac{1}{2}|\grad w|^2+\frac{1}{2(p-1)}w^2-e^{-\frac{(p+1)s}{p-1}}s^{\frac{2a}{p-1}}   F(\p w)\Big)\w \y, \label{18jan1111}\\
&&\no\\
L_0(w(s),s)&=&E(w(s),s)-\frac1{s\sqrt{s}}
\iint w^2\w\y,\label{5jan1111}
 %K(w(s),s)&=&H(w(s),s)+\frac{b}{2s^{b+1}+2s}L(w(s),s),\nonumber\\
%N(w(s),s)&=&  K(w(s),s)+\frac{\sigma}{s^{b-1}},\no
\end{eqnarray}
where
\begin{equation}\label{defF}
 F(u)=\int_{0}^{u}f(v){\mathrm{d}}v=\int_0^u|v|^{p-1}v \log^{{a}}(v^2  + 2)\v.
\end{equation}
% $F$ is defined by  \eqref{defF}.
%\begin{equation}\label{defF}
% F(u)=\int_{0}^{u}f(v){\mathrm{d}}v=\int_0^u|v|^{p-1}v \log^{{a}}(v^2  + 2)\v.
%\end{equation}
%where $F$  is defined by %\eqref{defF}.
%As we see above
Moreover,  for all
 $s\ge  \max (- \log T,1 )$,  we define the functional
\begin{equation}\label{10dec2}
L(w(s),s)=\exp\Big(\frac{p+3}{\sqrt{s}}\Big) L_0(w(s),s)+\frac{\theta}{s^{\frac34}},
\end{equation}
where $\theta$ is a sufficiently large constant that will be determined later.
We derive that the functional  $L(w(s),s)$ is a decreasing 
  functional  of time  for equation (\ref{A}),  provided that $s$ is  large enough.
Clearly, by  \eqref{18jan1111}, \eqref{5jan1111} and \eqref{10dec2},  the  functional $L(w(s),s)$  is a small perturbation of the   natural energy $E(w(s),s)$.

\medskip

Here is  the statement of  our main theorem in this paper.
 \begin{thm}[A Lyapunov functional in similarity variables]\label{t1}
Consider   $u $    a solution of ({\ref{gen}}),  with blow-up time
$T>0$.
% and  $x_0$  a non characteristic point.
Then, there exists $t_1\in [0,T) $ such that, % the  solution of equation $\eqref{A}$ satisfying
 %$(w_{x_0},\partial_{s}w_{x_0})\in C([-\log(T(x_0)),\,\,+\infty),\H)$, it holds that $L(w_{x_0}(s),s)$ defined in \eqref{10dec2} satisfies the following inequality,
  for all  $s\ge  -\log(T-t_1)$ and $x_0\in \er^N$,  we have
\begin{equation}\label{t1lyap}
 L(w(s+1),s+1)-L(w(s),s) \leq -\frac1{2} \int_{s}^{s+1}\iint (\partial_{s}w)^2\w\y \t,
\end{equation}
where  $w=w_{x_0,T}$ is defined in \eqref{scaling}.
%Moreover,  there exists $S_2\ge S_1$ such that, 
%  \begin{equation}\label{t1positive}
%L(w_{x_0}(s),s)\ge 0, \ \ \forall s\ge \max(-\log(T(x_0)),S_2).
%\end{equation}
\end{thm}

%\medskip

\begin{nb}
  We choose to put forward this result proving the existence of a Lyapunov functional and
  state it as the first result of our paper (namely Theorem \ref{t1}),
  mainly because we consider it as the crucial  step in our argument,
  and also because its proof  is far from being trivial.
% We have chosen to present our main  result as Theorem \ref{t1},  since
%   the existence of a Lyapunov functional in similarity variables is far from being trivial   and it  represents the crucial  step
%   in this paper.
\end{nb}

%
%
%
%\begin{nb}
%Let us note that   in   our paper \cite{HZjmaa2020} in  the hyperbolic case,  our 
% method is valid  only in one dimensional case
%and breaks down  in higher dimensional case. However,
%our method for this parabolic case \eqref{gen} in this paper  is valid in any dimensions.
%{\bf {jai ajoute une phrase }}
%Furthemore,  following the same strategy  of this paper,  we can  extend the result obtained in  \cite{HZjmaa2020}  
%   to  higher dimensional case. 
%This will be the subject of a forthcoming
%paper.
%\end{nb}

\medskip

The existence of this Lyapunov functional  $ L(w(s),s)$ together with a blow-up criterion for
equation \eqref{A} make  a crucial step in the derivation of the blow-up
rate for equation \eqref{gen}. Indeed, with the functional $ L(w(s),s)$, we are able to
adapt the analysis performed in  
\cite{GKcpam85,GKIuMJ87,GKCPAM89}
 for equation \eqref{NLH} and obtain the following result:

\begin{thm}
[Blow-up rate for equation \eqref{gen}]\label{t2}
Consider   $u $    a solution of ({\ref{gen}}),  with blow-up time
$T>0$.
% and  $x_0$  a non characteristic point.
Then, there exists $t_2\in [t_1,T) $ such that  for all
  $t\in [t_2,T)$,  we have
\begin{align}
\|u(t)\|_{L^{\infty}(\er^N)}\le K  (T - t)^{-\frac{1}{p-1}} (-\log (T - t))^{-\frac{a}{p-1}},\label{Heatest}
\end{align}
where 
%$\psi_{T(x_0)}(t)=(T(x_0)-t)^{-\frac2{p-1}}\Big(-\log (T(x_0)-t)\Big)^{-\frac{a}{p-1}}$, \\
$K=K(p,  a, T, t_2,\|u(\tilde {t_2})\|_{
L^{\infty}})$, for some $\tilde{t_2}\in [0,t_2)$. % $\psi_{T}(t)$ is defined in \eqref{psi}.
\end{thm}
\begin{nb}
Note that the blow-up rate in this upper bound is sharp, since we have
from a simple comparison argument the following lower bound:
\[
\|u(t)\|_{L^{\infty}(\er^N)} \ge v_T(t)\sim \kappa_a  (T - t)^{-\frac{1}{p-1}} (-\log (T - t))^{-\frac{a}{p-1}},
\]
where the last equivalence was given in \eqref{equivv}.
\end{nb}

%
%\begin{nb}
%As in the pure power nonlinearity  case  \eqref{NLH}, the proof of Theorem \ref{t2} relies on four ideas (the existence of a Lyapunov functional and   interpolation in Sobolev spaces). It happens that adapting the proof of
%\cite{MZimrn05} given in the  pure power nonlinearity  case \eqref{NLH} is straightforward. Therefore, we only present the key argument dedicated to the control of the 4th term in \eqref{perb4}, and refer to \cite{MZajm03, MZimrn05, MZma05}  for the treatment 
%of the terms
%appearing in the definition of $E_1(w(s),s)$ defined in \eqref{5jan1} and refer to 
% \cite{HZjhde12, HZnonl12,H1, omar1, omar2} 
% for the  control of the  three first  terms of \eqref{perb4}
%  for the rest of the proof.
%\end{nb}
%

\begin{nb}
Let us remark that we can obtain the same  blow-up rate for  the  more general equation 
\begin{equation}\label{NLWlog}
\partial_t u =\partial_x^2 u+|u|^{p-1}u\log^a(2+u^2)+k(u),\,\,\,(x,t)\in \R \times [0,T),
 \end{equation}
under the assumption  that $|k(u)|\leq M(1+|u|^{p}\log^b(2+u^2))$,   for some $M > 0$ and $b<a-1$. 
More precisely,   under this hypothesis,   we can construct  a suitable Lyapunov
functional for  this equation. Then, we  can prove a similar result to \eqref{Heatest}. However, the case where $a-1\le b<a$ seems 
to be out of  reach of our technics, though we think  we may obtain the same rate as in the unperturbed  case.
%In order to keep our analysis clear, we  restrict our analysis to  the  problem related to equation \eqref{gen}.
% The adaptation to the case
%is straightforward from the techniques of \cite{omar1}.
\end{nb}

%%%%%%%%%%%%%%%%%%%%%%%%%%%%%%
%%%%%%%%%%%%%%%%%%%%%%%%%%%%%%%%%%%%%%%%%%%%%%%%%%%%%%%%%%%%%%%%%%%%%%%%%%%%%%%%%%%%%
%%%%%%%%%%%%%%%%%%%%%%%%%%
%%%%%%%%%%%%%%%%%%%%%%%%%
%%%%%%%%%%%%%%%%%%%%%%%%%%%%%%%%%%%%%%%%%%%%%%%%%%%%%%%%%%%%%%%%%%%%%%%%%%

%
%
%{\bf Acknowledgment:} 
%The authors would like  thank the
%reviewers for their  valuable comments which undoubtedly helped us to improve the
%presentation of our results.

\medskip

%%%%%%%%%%%%%%%%%%%%%%%%%%%%%%%%%%%%%%%%%%%%%%%%%%%
%%%%%%%%%%%%%%%%%%%%%%%%%%%%%%%%%%%%%%%%%%%%%%%%%%%
%%%%%%%%%%%%%%%%%%%%%%%%%%%%%%%%%%%%%%%%%%%%%%%%%%%
%%%%%%%%%%%%%%%%%%%%%%%%%%%%%%%%%%%%%%%%%%%%%%%%%%%
This paper is organized as follows: In Section \ref{section2}, we obtain a rough control 
of the solution $w$. In Section \ref{section3},   thanks to that result, we prove that the 
functional $L(w(s),s)$  is a Lyapunov functional for equation \eqref{A}. Thus, we get Theorem \ref{t1}. Finally,  by
applying this last theorem,  we
 give the proof of
% are in position to  prove
Theorem \ref{t2}.

\medskip

Throughout this paper,
$C$  denotes a  generic positive constant
 depending only on $p,N$  and $a,$  which may vary from line to
 line. As for $M$, it will be used for constants depending on initial
 data, in addition to $p$, $N$ and $a$. We may also use $K_1$, $K_2$,
 $K_3$... $M_1$, $M_2$, $M_3$... $Q_1$, $Q_2$, $Q_3$ for constants
 having the same dependence as $M$. If necessary, we may write
 explicitly the dependence of the constants we use.
 % Note that
% Similarly,  we will use $M$ for  positive constants depending  on $p,N, a$  and initial data. 
% Additionally, we will often use
%  $K_1$, $K_2$, $K_3$... $M_1$, $M_2$, $M_3$... $Q_1$, $Q_2$, $Q_3$... 
%  as  positive constants depending  on $p,N, a$  and 
% the  initial data.  Nevertheless,  when it is necessary, we will make explicit the dependence of the constant  on the parameters of the problem.
Moreover, we denote 
 by $\mathbf{B}_R$  the open ball in $\er^N$ with center $0$ and
 radius $R$. Finally, note that we use the notation $f(s)\sim g(s)$ when
$\displaystyle{\lim_{s\to \infty}\frac{f(s)}{g(s)}=1}$.
\section{A polynomial bound for  solutions  of equation  \eqref{A}} \label{section2}
This section is devoted to the derivation of a  polynomial bound for a global   solution of equation  \eqref{A}.
% of rough estimate  for  solution of \eqref{A}.
  More precisely, this is the aim of this section.

\begin{prop}\label{prop21}
Let $R>0$.  Consider $w$   a global  solution of \eqref{A}. Then,     there exist
$\widehat{S}_1=\widehat{S}_1(a,p,N,R) \ge 1$ and  $\mu=\mu(a,p,N,R)>0$  such that,     for all $s\geq  \widehat{s}_1=\max (-\log T,\widehat{S}_1)$, 
we have
\begin{equation}\label{feb19}
\|  w(s )\|_{H^{1}(\mathbf{B}_R)}   \leq K_1(R) s^{\mu},
\end{equation}
 where  $K_1$ depends on $ p, a, N, R$  and
$\|w( \widehat{s}_1)\|_{H^1}$.
\end{prop}
%\begin{nb}
%Since the equation \eqref{A} is parabolic, the solution $w$ is smooth for $s\ge \widehat{s_1}$. Moreover, 
%$\|w( \widehat{s}_1)\|_{H^1}\le C \|w( \widehat{s}_1)\|_{L^{\infty}}$.
%\end{nb}

\bigskip

\begin{nb}\label{polybound}
By using  the Sobolev's embedding and the above proposition, we can deduce  that  for all $r\in [2,2^*)$, where $2^*=\frac{2N}{N-2}$, if $N\ge 3$ and $2^*=\infty$,  if $N=2$:
\begin{equation}\label{16dec5bis}
\|w(s)\|_{L^{r}(\mathbf{B}_R)} \le K_2(R)s^{\mu},\quad  \textrm{for all }\ \   
s\geq  \widehat{s}_1=\max (-\log T,\widehat{S}_1),
\end{equation}
where  $K_2(R)$ depends on $ p, a, N, R$  and
$\|w( \widehat{s}_1)\|_{H^1}$.
\end{nb}

In order to prove this proposition, we need to construct a  Lyapunov functional for equation \eqref{A}. 
%  and  conclude  a rough polynomial estimate related to the  the norm of $(w(s), \partial_s w(s))$, 
Accordingly,  we start by recalling from \eqref{18jan1111} the
following functional
\begin{equation}\label{AAA15dec}
E(w(s),s)=\iint \Big(\frac{1}{2}|\grad w|^2+\frac{1}{2(p-1)}w^2-e^{-\frac{(p+1)s}{p-1}}s^{\frac{2a}{p-1}}   F(\p w)\Big)\w \y,
\end{equation}
where $F$ is given by \eqref{defF}.
Then, we introduce the following 
functionals:
\begin{eqnarray}
%E(w(s),s)\!\!\!&=&\!\!\!\!\iint \Big(\frac{1}{2}|\grad w|^2+\frac{1}{2(p-1)}w^2-e^{-\frac{(p+1)s}{p-1}}s^{\frac{2a}{p-1}}   F(\p w)\Big)\w \y, \label{AAA15dec}\\
%&&\no\\
J(w(s),s)&=&-\frac1{2s}
\iint w^2\w\y,\label{5jan1}\\
H_{m}(w(s),s)&=&E(w(s),s)+m J(w(s),s),\label{F0}
 %K(w(s),s)&=&H(w(s),s)+\frac{b}{2s^{b+1}+2s}L(w(s),s),\nonumbe\\
%N(w(s),s)&=&  K(w(s),s)+\frac{\sigma}{s^{b-1}},\no
\end{eqnarray}
where
%$F$ is given by \eqref{defF}  and
$m >0$ is a sufficiently large constant that will be fixed later.
%where $F$  is defined by %\eqref{defF}.

\medskip

%As we see above
In fact, the main  target of this section is to prove, for some $\m_0$ large enough,
that the energy $H_{\m_0}(w(s),s)$ satisfies  the  following inequality: 
\begin{equation}\label{Heta0}
\frac{d}{ds}H_{\m_0}(w(s),s)\le - \frac12 \ibint (\partial_{s}w)^2{\w}\y
+\frac{\m_0 (p+3)}{2s}H_{\m_0}(w(s),s)
+  C e^{-s},
 \end{equation}
which implies  that 
$H_{m _0} (w(s),s)$ satisfies the following polynomial estimate:
\begin{equation}\label{Heta1}
H_{\m_0}(w(s),s))\le  A_0 s^{\mu_0},
 \end{equation}
for some $A_0>0$ and  $\mu_0>0$.% and 
%\begin{equation}\label{Heta2}
% \|w(s)\|_{H^{1}((-1,1))}+ \|\partial_s
%w(s)\|_{L^{2}((-1,1))} \le 
%  K s^{\mu_1},
% \end{equation}
%for some $\mu_1>0$.%
%\begin{nb}
%By exploiting  
%the  key inequality, we can construct a Lyaponov functional  as it will be showed in Proposition \ref{proplyap}. However this 
%where we see  that $N_m(w(s),s)$ sa point is a multiplier,
%$m(s)=\exp\Big(\frac{p+3}{\sqrt{s}}\Big) $ 
%%which is crucial for our proof. We note that $m(s)\sim 1$  as $s\to \infty$
%%
%
%
%\end{nb}
%%%%%%%%%%%%%%%%%%%%%%%%%%%%%%%%%%%%%%%%%%%%%%%%%%
%%%%%%%%%%%%%%%%%%%%%%%%%%%%%%%%%%%%%%%%%%%%%%%%%%
%%%%%%%%%%%%%%%%%%%%%%%%%%%%%%%%%%%%%%%%%%%%%%%%%%
%%%%%%%%%%%%%%%%%%%%%%%%%%%%%%%%%%%%%%%%

\subsection{Classical energy estimates}
In this subsection, we state two lemmas which are crucial for the construction of a  Lyapunov functional.
 We begin with bounding
the time derivative of   $E(w(s),s)$  in the following lemma:
\begin{lem}\label{lem22}  For all   
  $s \geq \T$, we have 
\begin{align}\label{E011}
\frac{d}{ds}E(w(s),s)\le& -  \frac{1}{2}\ibint (\partial_{s}w)^2\w\y
+  \frac{C}{s^{a+1}}\ibint |w|^{p+1}\log^a(2+\p^2w^2)\w \y
%\\ &
+\Sigma_{1}(s),%\no
%&&+ \frac{C}{s^2}\iint  (\partial_y w)^2(1-y^2)\w \y+\frac{C}{s^2}\iint  w^2\w \y  +C e^{-2s}.
 \end{align}
where $\Sigma_{1}(s)$  satisfies
\begin{align}\label{E011bis}
\Sigma_{1}(s)\le \frac{C}{s^2}\ibint  w^2\w \y  +C e^{-s}.
 \end{align}
\end{lem}
\begin{pf} Consider   $s \geq \T$. Multiplying $\eqref{A}$ by $\partial_{s} w\!\ \w$ and integrating over  $\er^N,$ we obtain
\begin{align}
\frac{d}{ds}E(w(s),s)=& -  \ibint (\partial_{s}w)^2\w\y+\underbrace{\frac{a}{(p-1)s}\ibint w\partial_{s}w\w\y}_{\Sigma^1_{1}(s)}\label{E00}\\
&+\underbrace{\frac{p+1}{p-1}
e^{-\frac{(p+1)s}{p-1}}s^{\frac{2a}{p-1}}\ibint\big( F(\p w)-\frac{\p wf(\p w)}{p+1}\big)\w \y}_{\Sigma^2_{1}(s)}\no\\
&\underbrace{-\frac{2a}{p-1}e^{-\frac{(p+1)s}{p-1}}s^{\frac{2a}{p-1}-1}\ibint \big( F(\p w)-\frac{\p wf(\p w)}{2}\big)\w \y}_{\Sigma^3_{1}(s)}.\no
 \end{align}

Now, we control the terms $\Sigma_{1}^{1}(s)$, $\Sigma_{1}^{2}(s)$ and  $ \Sigma_{1}^{3}(s)$. 
  By using  the following  basic inequality
\begin{equation}\label{basic1}
ab\le \varepsilon a^2+\frac1{\varepsilon}b^2, \ \forall \varepsilon >0,
\end{equation}
  we  write
\begin{equation}\label{sigma13}
\Sigma_{1}^{1}(s)\leq \frac1{2} \ibint (\partial_{s}w)^2\w \y+ \frac{C}{s^2}\ibint  w^2\w \y.
\end{equation}
Let us introduce the functions $F_1$ and $F_2$  defined by:
\begin{equation}
F_1(x)= -\frac{ 2a} {(p+1)^2}{| x|^{p+1}}\log^{{a-1}}(2+x^2  ),\label{defF2}
\end{equation}
and 
\begin{equation}\label{defF123}
F_2(x)=F(x)-\frac{xf(x)}{p+1}-F_1(x).
\end{equation}
By  the expressions of $F_1$,  $F_2$ given by   \eqref{defF2} and \eqref{defF123}
 and the estimates   \eqref{equiv2}  and  \eqref{equiv3}, we obtain
\begin{equation}\label{id2}
 F(\p w)-\frac{\p wf(\p w)}{p+1}\le   C+C \frac{ \p w}{s}f(\p w),
\end{equation}
which implies  %for all $s\ge \T$,
\begin{equation}\label{sigma11}
\Sigma_{1}^{2}(s)\le C
e^{-\frac{(p+1)s}{p-1}}s^{\frac{2a}{p-1}-1}
\ibint  \p wf(\p w) 
\w \y+  C e^{-s}.
\end{equation}
From the expression of $\p= \ps$ defined in \eqref{defphi},   we have% for all $s\ge \T$, 
\begin{equation}\label{id1}
e^{-\frac{(p+1)s}{p-1}}s^{\frac{2a}{p-1}}  \p wf(\p w)=\frac1{s^{a}}|w|^{p+1}\log^a(2+\p^2w^2).  
\end{equation}
Thus, using \eqref{sigma11} and \eqref{id1}, we obtain% for all $s\ge\T$,
\begin{equation}\label{sigma111}
\Sigma_{1}^{2}(s)\le
\frac{C}{s^{a+1}}\ibint |w|^{p+1}\log^a(2+\p^2w^2)\w \y+  C e^{-s}.
\end{equation}
Similarly, by
 \eqref{equiv1} and \eqref{id1},  we easily obtain %, for all $s\ge \T$,
\begin{equation}\label{sigma12}
\Sigma_{1}^{3}(s)\le \frac{C}{s^{a+1}}\ibint |w|^{p+1}\log^a(2+\p^2 w^2)\w \y+  C e^{-s}.
\end{equation}
%
%\medskip
%%In the sequel, we shall intensively and implicitly use this
The results \eqref{E011} and \eqref{E011bis} follows  immediately from  \eqref{E00}, \eqref{sigma13},  \eqref{sigma111}  and \eqref{sigma12},
which ends the proof of Lemma \ref{lem22}.
\end{pf}

\vspace{-0.5cm}

\begin{nb}
By showing the estimate proved in Lemma \ref{lem22}, related to the so called  natural functional $E(w(s),s)$, we have  some
nonnegative  terms  in the right-hand side of \eqref{E011} and this  does not allow to construct a decreasing
 functional (unlike the case of a  pure power nonlinearity). 
The main problem   is related to  the nonlinear term 
$$\frac{1}{s^{a+1}}\ibint |w|^{p+1}\log^a(2+\p^2(s)w^2)\w \y=\frac{1}{s}\ibint w
e^{-\frac{ps}{p-1}}s^{\frac{a}{p-1}} f(\ps w)\w \y.$$     
%Even i'ts  a lower order,  this term
%%$has a clear effect because an estimate of type (
 To overcome this problem, we adapt the strategy  used in  \cite{HZjhde12, HZnonl12,H1, omar1, omar2,V}. Indeed, by using the identity obtained 
 by multiplying  equation \eqref{gen} by $w\w$, 
then integrating over $\er^N$, we can 
  introduce a  new functional $
H_m(w(s),s)$ defined in \eqref{F0},  where $m>0$ is sufficiently large and
 will be fixed such that $H_{m}(w(s),s)$
satisfies a differential inequality similar to  \eqref{jardin1}.  %This choice is due to the presence of  
%%. Unfortunately
\end{nb}

 \medskip
%%%%%%%%%%%%%%%%%%%%%%%%%%%%%%%%%%%%%%%%%%%%%%%%%%
%%%%%%%%%%%%%%%%%%%%%%%%%%%%%%%%%%%%%%%%%%%%%%%%%%
%%%%%%%%%%%%%%%%%%%%%%%%%%%%%%%%%%%%%%%%%%%%%%%%%%
%%%%%%%%%%%%%%%%%%%%%%%%%%%%%%%%%%%%%%%%

%  $s \geq \max(-\log T, 1)$
We will prove the following estimate on the functional $J(w(s),s)$. 
 \begin{lem}\label{LemJ_0}
For  all $s \geq \T$, we have 
\begin{align}\label{6nov2018}
\frac{d}{ds}J(w(s),s)\ \  \le&\ \ \frac{p+3}{2s}E(w(s),s)
-\frac{p-1}{4s}\ibint  |\grad w|^2\w\y
- \frac{1}{4s} \ibint  w^2\w\y\no\\
&
%-\frac{m(p+1)}{2(p-1)s} \ibint  w^2\w\y
-\frac{p-1}{2(p+1)s^{a+1}}\ibint |w|^{p+1}\log^a(2+\p^2 w^2)  
\w\y+\Sigma_2(s),
\end{align}
where  $\Sigma_2(s)$ satisfies
  \begin{align}\label{DEC1999}
  \Sigma_{2}(s)\leq
&\frac{C}{s^{a+2}}\ibint |w|^{p+1}\log^a(2+\p^2w^2)\w \y+\frac{C}{s^2}\ibint  w^2\w\y+  C e^{-s}.
\end{align}
\end{lem}
\begin{pf} Consider   $s \geq \T$.
Note that $J(w(s),s)$ is a differentiable function and   that we get  
%for all  $s \geq \T$, 
\begin{equation*}\label{j1}
\frac{d}{ds}J(w(s),s)=-\frac1{s} \ibint   w\partial_{s}w\w\y +\frac{1}{2s^2} \ibint   w^2\w \y.
\end{equation*}
From equation  $\eqref{A}$ and the identity \eqref{id1}, we conclude
\begin{align}\label{j2}
\frac{d}{ds}J(w(s),s)=&\J \ibint |\grad w|^2\w \y+ \frac{1}{(p-1)s} 
  \ibint  w^2\w\y\no \\
&-\frac1{s^{a+1}}\iint |w|^{p+1}\log^a(2+\p^2w^2)\w \y
+ \frac{1}{2s^2}\big(1-\frac{2a}{p-1}\big)   \ibint  w^2\w\y.\no
 \end{align}
According to the expressions of $E(w(s),s)$,  $\ps$ defined in \eqref{AAA15dec} and  \eqref{defphi} and the identity \eqref{id1} 
with some straightforward computation,  we obtain  \eqref{6nov2018} where
 \begin{equation}\label{R0}
\Sigma_2(s)=\Sigma^1_2(s)+\Sigma_{2}^{2}(s),
\end{equation} 
and 
\begin{align*}
\Sigma_{2}^1(s)=&
\frac{p+3}2e^{-\frac{(p+1)s}{p-1}}s^{\frac{2a}{p-1}-1} 
  \ibint \Big( F(\p w)-\frac{\p wf(\p w)}{p+1}\Big)\w\y,  \no \\
\Sigma_{2}^2(s)=& \frac{1}{2s^2}\big(1-\frac{2a}{p-1}\big)   \ibint  w^2\w\y.
\no 
 \end{align*}
%We are going now to estimate  the  different terms of \eqref{R0}.
%Using the fact that for all $s \geq \max(s_{0},1)$, $\Big|\frac{b}{s}+\frac{p+3}{2s^{b}}\Big|\leq \frac{C}{s},$
Thanks to \eqref{id1} and  \eqref{id2},  we deduce % that  for all $s\ge \T$
\begin{equation}\label{R1}
\Sigma_{2}^{1}(s)\le
\frac{C}{s^{a+2}}\ibint |w|^{p+1}\log^a(2+\p^2w^2)\w \y+  C e^{-s}.
\end{equation}
Hence, collecting \eqref{R0} and  \eqref{R1}, one easily obtains that $\Sigma_{2}(s)$ satisfies \eqref{DEC1999},
 which ends the proof of Lemma \ref{LemJ_0}.
\end{pf}
\subsection{Existence of a decreasing functional for equation \eqref{A}}
In this subsection, by using Lemmas  \ref{lem22} and \ref{LemJ_0}, we will construct a decreasing
functional for equation (\ref{A}). Let us define the following functional:
\begin{equation}\label{lyap1}
N_{m}(w(s),s)=s^{-\frac{m(p+3)}{2}}H_{m}(w(s),s)+ A (m) e^{-s},
\end{equation} where 
$H_{m}(w(s),s)$
 is defined in \eqref{F0}, and  $m$ together with  $A=A(m)$ are   constants that  will be determined later.

\medskip

We now state the following proposition:
\begin{prop}\label{proplyap}
There exist $m_0>1$, $A (m_0) >0$,  $S_{1}\geq 1$ and $\lambda_1>0$, such that for all $s=s_1 \geq \TT$, 
we have 
\begin{align}\label{DEC101} 
N_{m_0}(w(s+1),s+1)-N_{m_0}(w(s),s)\le& - \frac{\lambda_1}{s^b}\ia\ibint\!\!
(\partial_{s}w)^2\w
\y\t\\
&\!\!\!-\frac{\lambda_1}{s^{b+1}} \ia\ibint  |\grad w|^2\w\y\t\no\\
& -\frac{\lambda_1}{s^{a+b+1}} 
\ia \ibint |w|^{p+1}\log^a(2+\p^2w^2)  
\w\y\t\no\\
&-\frac{\lambda_1}{s^{b+1}}\ia\ibint\!\!
w^2\w
\y\t,\qquad\no
%&-\frac{\lambda_1}{s^{b+1}}\ia\iint\!\!(\partial_sw)^2\w\y\t
\end{align}
where
\begin{equation}\label{defb}
b=\frac{m_0(p+3)}{2}.
\end{equation}
Moreover, there exists
 $S_{2}\geq S_1$ such that for all $s \geq \TTT$, 
we have 
 \begin{equation}\label{posi1}
N_{m_0}(w(s),s)\geq -1.
\end{equation}
\end{prop}

\begin{pf}
From  the definition of $H_m(w(s),s)$ given  in \eqref{F0},   Lemmas  \ref{lem22} and \ref{LemJ_0},  we can write for all $s\ge \T$,
\begin{align}\label{DEC9}
\frac{d}{ds}H_{m}(w(s),s)\le & \frac{m(p+3)}{2s}H_{m}(w(s),s) -\frac12\ibint   (\partial_{s}w)^2\w \y\\
&-\Big(\frac{m(p-1)}{2(p+1)}-C_0-\frac{C_0 m}{s}\Big)
\frac{1}{s^{a+1}}
\ibint |w|^{p+1}\log^a(2+\p^2w^2)  
\w\y\no\\
&-\frac{m(p-1)}{4s}\ibint   |\grad w|^2\w \y-\Big( \frac{m}{4s}-\frac{C_0m}{s^2}-\frac{C_0}{s^2}
\Big) \ibint  w^2\w\y  \no \\
&+(C_0 m+C_0) e^{-s},\no
\end{align}
where $C_0$  stands for some universal constant depending only on $N,$ $p$ and $a$.
We first  choose    $m_0$ such that $\frac{m_0(p-1)}{4(p+1)}-C_0=0$,  so 
$$\frac{m_0(p-1)}{2(p+1)}-C_0-\frac{C_0m_0}{s}=m_0\Big(\frac{p-1}{4(p+1)}-\frac{C_0}{s}\Big).$$
We now choose    $S_1=S_1(a,p,N)$ large enough ($S_1\ge 1$), so that for all $s\ge S_1$, we have
\begin{align*}
\frac{p-1}{8(p+1)}-\frac{C_0}{s} \ge 0,\qquad  \qquad \qquad
\frac{m_0}{8}-\frac{C_0m_0}{s}-\frac{C_0}{s}\ge 0.
\end{align*}
Then, we deduce that  for all $s\ge \TT$,
\begin{align}\label{DEC911}
\frac{d}{ds}H_{m_0}(w(s),s)\ \ \le & \ \ \frac{m_0(p+3)}{2s}H_{m_0}(w(s),s) -\frac12\ibint   (\partial_{s}w)^2\w \y\\
&-\frac{\lambda_0}{s}\ibint   |\grad w|^2\w \y-\frac{\lambda_0}{s} \ibint  w^2\w\y  \no \\
&-\frac{\lambda_0}{s^{a+1}}
\ibint |w|^{p+1}\log^a(2+\p^2w^2)  
\w\y\no\\
&+(C_0 m_0+C_0) e^{-s},\no
\end{align}
 where $\lambda_0=\inf( \frac{m_0}{8},
\frac{m_0(p-1)}{4(p+1)} )$.

By using the definition of $N_{m_0}(w(s),s)$ given  in  \eqref{lyap1} together with  the estimate \eqref{DEC911}, we  easily prove that  for all  $s\ge \TT$,
\begin{align}\label{M11}
\frac{d}{ds}N_{m_0}(w(s),s)\ \ \le & \ \ -\frac1{2s^b}\ibint   (\partial_{s}w)^2\w \y\\
&-\frac{\lambda_0}{s^{b+1}}\ibint   |\grad w|^2\w \y-\frac{\lambda_0}{s^{b+1}} \ibint  w^2\w\y  \no \\
&-\frac{\lambda_0}{s^{a+b+1}}
\ibint |w|^{p+1}\log^a(2+\p^2w^2)  
\w\y\no\\
&-e^{-s}\Big(A(m_0) -C_0(m_0+1)\frac{1}{s^b}\Big).\no
\end{align}
We now choose $A(m_0) =C_0(m_0+1){S_1}^{-b}$,  so we have
 \begin{equation}\label{11jan1}
A(m_0) -\frac{C_0(m_0+1)}{s^b}\ge 0,\qquad  \forall  s\ge S_1.  
\end{equation}
By  integrating  in time between $s$ and $s+1$
the  inequality
(\ref{M11}) and using \eqref{11jan1}, we easily obtain
(\ref{DEC101}).
This concludes the proof of the first part of Proposition \ref{proplyap}.

\medskip

 We  prove    \eqref{posi1} here. 
%  The argument is the
%same as in the corresponding part in \cite{
%HZjhde12, HZnonl12,H1, omar1, omar2}. We write the
%proof for completeness.
 Arguing by contradiction, we assume that
there exists   
 $\tilde{s_1} \geq \TTT$ such that $N_{m_0}(w(\tilde{s_1}),\tilde{s_1})<-1$, where $S_2\ge S_1$  is large enough.

\bigskip

Now, we 
 consider $$I(w(s),s)=s^{-b}\ibint w^2  \w\y,\qquad   \forall s \geq \T,$$
where $b$ is defined in \eqref{defb}. Thanks to \eqref{6nov2018} and \eqref{F0}, we have  for  all $s \geq \T$
\begin{align}\label{6nov20181}
\frac{d}{ds}I(w(s),s)\  \ge&\  -(p+3)s^{-b}H_{m_0}(w(s),s)
+ \frac{1}{2s^{b}}(1-\frac{C_1}{s}) \ibint  w^2\w\y\no\\
&\ 
%-\frac{m(p+1)}{2(p-1)s} \ibint  w^2\w\y
+\frac{p-1}{(p+1)s^{a+b}}(1-\frac{C_1}{s})\ibint |w|^{p+1}\log^a(2+\p^2 w^2)  
\w\y.
\end{align}
Let us  choose   $S_2=S_2(a,p,N)$  is large enough,  such that $1-\frac{C_1}{S_2}\ge \frac12$. So, we write  for  all $s \geq \TTT$ 
\begin{equation}\label{1jan2020}
\frac{d}{ds}I(w(s),s) \ge -(p+3)N_{m_0}(w(s),s)
+\frac{p-1}{2(p+1)s^{a+b}}\ibint |w|^{p+1}\log^a(2+\p^2 w^2)  
\w\y.
\end{equation}
 Since the energy $N_{m_0}(w(s),s)$ decreases in time, we
have $N_{m_0}(w(s),s)<-1$, for all $s\ge \tilde{s_1}$. Then, for all $s\ge \tilde{s_1}$
\begin{align}\label{2jan20201}
\frac{d}{ds}I(w(s),s) \ge p+3
+\frac{p-1}{2(p+1)s^{a+b}}\ibint |w|^{p+1}\log^a(2+\p^2 w^2)  
\w\y.%\no \\
%&\ \ \ + (p+3) \sigma (m_0) e^{-\frac{s}2}.
\end{align}
%
%
%To estimate the right-hand  side in the inequality \eqref{2jan20201}, we consider two cases: \\
%{\bf Case 1:} the  case where $a\ge 0$.\\
%
%
%
%
%%We would like now to find  an estimate for the  term $\ibint |w|^{p+1}\log^a(2+\p^2 w^2)  
%%\w\y$
%We divide $\er^N$ into two parts
% \begin{equation}\label{27nov2}
%A_{1}(s)=\{y \in \er\,\,|\,\, \ps w^2(y,s)\leq  1\}\,\,{\rm and }\,\,A_{2}(s)=\{y \in \er
%\,\,|\,\, \ps  w^2(y,s)\ge  1\}.
%\end{equation}
%%Accordingly, we write  $\ibint |w|^{p+1}\log^a(2+\p^2 w^2)  
%%\w\y=\chi_1(s)+\chi_2(s)$, where
%%\begin{align}
%%\chi_1(s)=&
%%\int_{A_1(s)}   {|  w|^{p+1}}\log^{{a}}(2+\p^2 w^2  )\w \y,\label{130}\\
%%\chi_{2}(s)=&
%%\int_{A_2(s)}   {|  w|^{p+1}}\log^{{a}}(2+\p^2 w^2  )\w \y.\label{131}
%%\end{align}
%Note that, by using  the definition of the set $A_1(s)$ given  in \eqref{27nov2},  we get, for all $s\geq  \TTT,$
%\begin{equation}\label{13janv0}
%\frac{1}{s^a}\int_{A_2(s)}   {|  w|^{p+1}}\log^{{a}}(2+\p^2 w^2  )\w \y\ge 
%C\int_{A_2(s)}   {|  w|^{p+1}}\w \y,
%\end{equation}
%and
%\begin{equation}\label{13janv1}
%\int_{A_1(s)}   {|  w|^{p+1}}\w \y\le Ce^{-s}.
%\end{equation}
%Thanks to \eqref{equiv1} and  \eqref{equiv4bis},  we have for all $s\geq  s^*,$ for all $r\in (1,p),$
%\begin{equation}\label{13janv2}
%\frac{1}{s^a}\int_{\er^N}   {|  w|^{p+1}}\log^{{a}}(2+\p^2 w^2  ) \psi^2 \w \y\ge 
%C\int_{\er^N}   {|  w|^{r+1}} \psi^2\w \y-
% C.
%\end{equation}
Thanks to \eqref{equiv1},  \eqref{equiv4bis} and \eqref{id1},  we get  for all $s\geq  \tilde{ s_1}$ %for all $r\in (1,p),$
\begin{equation}\label{13janv2bis}
\frac{1}{s^a}\int_{\er^N}   {|  w|^{p+1}}\log^{{a}}(2+\p^2 w^2  )\w \y\ge 
C\int_{\er^N}   |  w|^{\frac{p+3}2}\w \y-
 C.
\end{equation}
Therefore, by using \eqref{2jan20201} and \eqref{13janv2bis}, there exist $\tilde S_2\ge S_2$ large enough such that $ p+2-\frac{C}{(\tilde S_2)^b}>0$, we have  for all $s\ge \max (\tilde{s_1},\tilde S_2)$
\begin{align}\label{13jan3}
\frac{d}{ds}I(w(s),s) \ge 1
+\frac{C}{s^{b}}\ibint |w|^{\frac{p+3}2}  
\w\y.%\no \\
%&\ \ \ + (p+3) \sigma (m_0) e^{-\frac{s}2}.
\end{align}
Thanks  to    Jensen's inequality, we infer
\begin{align}\label{2jan202000000}
\frac{d}{ds}I(w(s),s)\ge 1+ Cs^{\frac{b(p-1)}
4 }   \Big(I(w(s),s) \Big)^{\frac{p+3}4 }.%\no \\
%&\ \ \ + (p+3) \sigma (m_0) e^{-\frac{s}2}.
\end{align}
\noindent 
This quantity must then tend to $\infty$ in finite time, which is a contradiction.
   Thus \eqref{posi1} holds.
 This concludes  the  proof of Proposition \ref{proplyap}.
\end{pf}
\subsection{Proof of Proposition   \ref{19dec3bis1}.}
Based on Proposition \ref{proplyap},    a bootstrap argument given in \cite{Qu99amuc}, 
we  are able to  adapt the analysis
performed in  
 \cite{GMSiumj04},  to prove the following key proposition:

\begin{prop}\label{19dec3bis1} 
For all  $q \geq 2$,  $\varepsilon>0$ and $R>0$
there exist   $\varepsilon_1=\varepsilon_1(q,R)>0$,
$\mu_1(q,R,\vv)>0$ and 
 $S_{3}(q,R,\varepsilon) \geq S_2$ such that,  for  all $s \geq \max  (-\log T, S_3),$
 we have
\begin{equation*}
(A_{q,R,\vv})\ \ \int_s^{s+1} \|w(\tau )\|_{L^{\pp +1}(\mathbf{B}_R)}^{(\pp+1)q} {\mathrm{d}\tau}  \leq K_{3}(q,R,\varepsilon )s^{\mu_1(q,R,\vv)}, \qquad \forall \vv\in (0,\vvv],
\end{equation*}
where  $K_{3}(q,R,\varepsilon )$ depends on $ p, a, N,q, R,\varepsilon, s_1 =\TT$  and
$\|w( {s}_1)\|_{H^1}$.
\end{prop}

To prove Proposition \ref{19dec3bis1}, we will proceed as in \cite{GMSiumj04}. In fact, by using
%As a consequence of 
Proposition \ref{proplyap},
 we easily  obtain the following Corollary:
% which summarizes the principle properties of $N_{m_0}(w(s),s)$ defined in  \eqref{lyap1} and  will be useful for getting  Proposition 
%\ref{19dec3}:
%  and  Proposition  \ref{prop21}:
\begin{coro}\label{19dec3} 
  For all $s\ge  \TTT,$  
 we have
\begin{equation}\label{2018N0}
-1\leq N_{m_0}(w(s),s) \leq K_4,
\end{equation}
\begin{equation}\label{18jan1bis}
 -K_5s^{b}\le 
H_{m_0}(w(s),s ) \leq K_5s^{b},
\end{equation}
\begin{equation}\label{18jan1}
 \int_{s}^{s+1}\iint\Big(|\grad w|^2+(\partial_sw)^{2}+w^2\Big) \w  {\mathrm{d}}y{\mathrm{d}}\tau\leq K_6s^{b+1},
\end{equation}
\begin{equation}\label{18jan10000}
\frac1{s^a}\int_{s}^{s+1}\int_{\er^N}   {|  w|^{p+1}}\log^{{a}}(2+\p^2 w^2  )\w \y
{\mathrm{d}}\tau\leq K_6s^{b+1},
\end{equation}
\begin{equation}\label{18jjjjj}
\iint w^2\w   \y\leq K_7s^{b+1},
%\|w(s)\|_{L_\rho^2(\mathbb{R}^n)}^2 \leq J_1,
\end{equation}
\begin{equation}\label{30ct1}
\frac1{s^a}\int_{\er^N}   {|  w|^{p+1}}\log^{{a}}(2+\p^2 w^2  )\w \y
\leq C\iint |\grad w|^2 \w  \y +K_8s^{b+1},
\end{equation}
\begin{equation}\label{30ct1bisbis}
 \iint |\grad w|^2 \w  \y \le 
\frac{C}{s^a}\int_{\er^N}   {|  w|^{p+1}}\log^{{a}}(2+\p^2 w^2  )\w \y
 +K_9s^{b+1},
\end{equation}
\begin{equation}\label{30ct2}
 \iint |\grad w|^2 \w  {\mathrm{d}}y
\leq Cs^{\frac{b+1}2}\sqrt{\iint (\partial_sw)^2 \w  {\mathrm{d}}y} +K_{10}s^{b+1},
\end{equation}
\begin{equation}\label{30ct24}
\int_{s}^{s+1}\Big( \iint |\grad w|^2 \w  {\mathrm{d}}y\Big)^2
\leq K_{11}s^{2b+2},
\end{equation}
\begin{equation}\label{3oct3}
\frac1{s^{2a}}\int_{s}^{s+1}\Big(\int_{\er^N}   {|  w|^{p+1}}\log^{{a}}(2+\p^2 w^2  )\w \y\Big)^2
{\mathrm{d}}\tau\leq K_{12}s^{2b+2},
\end{equation}
where  $b$ is defined in \eqref{defb} and where $K_4, K_5,K_6,... K_{12} $ depend on $ p, a, N,  s_1 =\TT$  and
$\|w( s_1)\|_{H^1}$.
\end{coro}

\begin{nb}\label{r1}
 Let us mention that, the estimates obtained  in the above corollary  are similar to the ones
obtained in the pure power case treated in  \cite{GMSiumj04} except for the following features:
\begin{itemize}
\item The presence of the  term $Ks^{b+1}$ instead of $K$.
\item In  some estimates, we have the term $F(u)$ instead  of $\frac{|u|^{p+1}}{p+1}$ 
in the pure power case.  We easily overcome this problem
 thanks to the fact that   $uf(u)-(p+1)\int_0^uf(v)
{\mathrm{d}}v\sim  
\frac{2a}{p+1}|u|^{p+1}\log^{a-1}(2+u^2)$, as $u\to \infty$.
\end{itemize}
\end{nb}

\bigskip

In order to prove  
Proposition   \ref{19dec3bis1},   we introduce the following local functional:
%which is a perturbed version of the function of \cite{GMSiumj04},
%  and  conclude  a rough polynomial estimate related to the  the norm of $(w(s), \partial_s w(s))$, 
\begin{equation}\label{Eloc}
\mathcal{E}_\psi(w(s),s)=\iint \Big(\frac{1}{2}|\grad w|^2+\frac{1}{2(p-1)}w^2-e^{-\frac{(p+1)s}{p-1}}s^{\frac{2a}{p-1}}   F(\p w)\Big)\psi^2(y)\w \y, 
%\mathcal{J}_\psi(w(s),s)&=&-\frac1{2s}
%\iint \psi^2 w^2\w\y,\label{5jan1bis}\\
%\mathcal{H}_{\psi,m }(w(s),s)&=&\mathcal{E}_\psi(w(s),s)+m \mathcal{J}_\psi(w(s),s),\label{F0bis}
 %K(w(s),s)&=&H(w(s),s)+\frac{b}{2s^{b+1}+2s}L(w(s),s),\nonumbe\\
%N(w(s),s)&=&  K(w(s),s)+\frac{\sigma}{s^{b-1}},\no
\end{equation}
where  $\psi \in
 \mathcal{C}_0^\infty(\mathbb{R}^N)$
% \mathcal{C}^2(\mathbb{R}^N)$
 %is a bounded function and 
%where $F$ is given by \eqref{defF}
%\noindent Let $R > 0$, we fix $\psi(y)$
 satisfies 
\begin{equation} \label{equ:psiy}
 0 \leq \psi(y) \leq 1, \quad \psi(y) = \left\{\begin{array}{lcl} 1 & \quad\text{on} &\quad \mathbf{B}_R \\
0 & \quad \text{on} & \quad \mathbb{R}^N\setminus \mathbf{B}_{2R} \end{array} \right.,
\end{equation}
where  $R>0$.
An argument similar to that in \cite{GMSiumj04}, implies  the following estimate: 
%We get the following bound on the local functional $\mathcal{E}_\psi$:
\begin{prop} \label{prop:upElc} %Let
 % $a, p, N$ be fixed and 
% $w$ be solution of equation \eqref{A}. 
There exist
 positive constants $ K_{13}=K_{13}(R)>0$  and 
 $S_{4}\geq S_2$ such that,  for  all $s \geq \max  (-\log T, S_4),$
 we have
\begin{equation}\label{Eloc1}
- K_{13}(R)s^{b+1} \leq \mathcal{E}_{\psi}(w(s),s) \leq  K_{13}(R)s^{b+1}, %\quad \forall s \geq \max  (-\log T, S_4),
\end{equation}
where 
$ K_{13}$
%, $\|\psi\|^2_{L^\infty}$, $\|\nabla \psi\|^2_{L^\infty}$
depends on $ p, a, N,  R,  s_1 =\TT$  and
$\|w( s_1)\|_{H^1}$.
\end{prop}

\begin{pf}Most of the steps of the proof  are the same as
in the pure power case  treated in  \cite{GMSiumj04}  and some others are more delicate.
For that reason, 
we leave the proof to Appendix \ref{ap:upELc}.
\end{pf}
%
%\medskip
%
With Proposition \ref{prop:upElc}, we are in a position to claim the following:
\begin{lem}\label{rema:boundLpW12} There exists a
 positive constant   $ K_{14}(R,\varepsilon )>0$   such that,  for  all $s \geq \max  (-\log T, S_4)$
 
% Let $a, p, N$ be fixed and $w$ be solution of equation \eqref{A}.  Then
% exist positive constants $ K_2$ such that 
% Then there exists $\tilde{s}_5 \geq \tilde{s}_3$ such that 
\begin{equation}\label{31oct123}
 \|w(s )\|_{L^{\pp +1}(\mathbf{B}_R)}^{(\pp+1)}   \leq K_{14}(R,\varepsilon ) \|\grad w \|^2_{L^2(\mathbf{B}_{2R})}+K_{14}(R,\varepsilon )s^{b+1}, \ \ \forall  \vv \in (0,p-1),
%  \ \ \forall s \geq \max  (-\log T, S_5).
\end{equation}
where 
$ K_{14}(R,\varepsilon )$
%, $\|\psi\|^2_{L^\infty}$, $\|\nabla \psi\|^2_{L^\infty}$
depends on $ p, a, N,  R, \varepsilon, s_1 =\TT$  and
$\|w( s_1)\|_{H^1}$.
\end{lem}
\begin{pf} From \eqref{Eloc1} and the definition of $\mathcal{E}_\psi$ in \eqref{Eloc}, we have
for all $ s \geq \max  (-\log T, S_4),$
\begin{equation}\label{Eloc3}
e^{-\frac{(p+1)s}{p-1}}s^{\frac{2a}{p-1}} \iint  F(\p w)\psi^2 \w \y \le C \iint |\grad w|^2\psi^2 \w \y+K_{13}(R)s^{b+1}.
\end{equation}
By exploiting  \eqref{equiv4bis}, we write  for all $s \geq \max  (-\log T, S_4),$
\begin{equation}\label{Eloc4}
\ibint |w|^{\pp+1}\psi^2\w \y \le C \iint |\grad w|^2\psi^2\w \y+K_{13}(R)s^{b+1}+C(\varepsilon)e^{-s}, \quad \forall  \vv \in (0,p-1). 
\end{equation}
Thus, \eqref{31oct123} follows from \eqref{Eloc4} and the property of $\psi$. This conclude the proof of Lemma \ref{rema:boundLpW12}
\end{pf}
By \eqref{31oct123}, the proof of estimate $(A_{q,R,\vv})$  is available when we have
\begin{equation}\label{equ:keyintW}
\int_s^{s+1} \|\grad w(\tau)\|_{L^2(\mathbf{B}_{R})}^{2q} \t \leq K_{15}(q,R,\vv )s^{\mu_2(q,R,\vv)},\qquad \forall s\geq \max  (-\log T, S_3), 
\end{equation}
for some $\mu_2(q,R,\vv)>(b+1)q$.
Note from \eqref{30ct24}  that \eqref{equ:keyintW} already holds in the case $q = 2$.
%\end{rema}

\medskip

 In order to derive \eqref{equ:keyintW} for all $q > 2$, we need the following result:
\begin{lem} \label{lemm:tmpq2} %Let $a, p, N$ be fixed and $w$ be solution of equation \eqref{A}. Then t
There exist
 positive constants $ K_{16}(R)>0$  and 
 $S_5\geq S_4$ such that,  we have
\begin{equation}\label{equ:boundLp123}
\|\grad w\|_{L^2(\mathbf{B}_R)}^{2} \leq C\| w \partial_s w \psi^2\|_{L^1(\mathbf{B}_{2R})}+K_{16}(R)s^{b+1} , \ \ \forall s \geq \max  (-\log T, S_5).
\end{equation}
\end{lem}
\begin{pf}
Multiplying equation \eqref{A} with $w\w\psi^2,$ integrating over $\mathbb{R}^N$ and using the definition of $\mathcal{E}_\psi(w(s),s)$  given in \eqref{Eloc},  we write
\begin{align}\label{1nov1}
\ibint  |\nabla w|^2\psi^2\w\y =& \frac4{p-1}\ibint  w\partial_sw \psi^2\w\y  +\underbrace{
\frac{2(p+3)}{p-1}\mathcal{E}_\psi(w(s),s)
}_{\Sigma^1_{2}(s)}
\\
&-\frac{2}{(p+1)s^{a}}\ibint |w|^{p+1}\log^a(2+\p^2 w^2)  
 \psi^2\w\y\no\\
&+\underbrace{ \frac8{p-1}\int_{\mathbb{R}^N}  w \nabla w . \nabla  \psi   \psi \w \y}_{\Sigma^2_{2}(s)}
\underbrace{-\frac1{p-1}(1+\frac{4a}{(p-1)s})\ibint   w^2\psi^2\w\y}_{\Sigma^3_{2}(s)}\no\\
%& - 4 s^{-b-1}\int_{\mathbb{R}^N} \psi w \nabla \psi . \nabla w \w \y. 
&+\underbrace{\frac{2(p+3)}{p-1}
e^{-\frac{(p+1)s}{p-1}}s^{\frac{2a}{p-1}}\ibint \big( F(\p w)-\frac{\p wf(\p w)}{p+1}\big)\ \psi^2w \y}_{\Sigma^5_{2}(s)}. \no
\end{align}
From \eqref{Eloc1}, \eqref{1fev6} and \eqref{18jjjjj} we infer
 for all  $s \geq \max  (-\log T, S_4),$
\begin{equation}\label{jan211}
\Sigma^1_{2}(s)+\Sigma^2_{2}(s)+\Sigma^3_{2}(s)\le  K_{17}(R)s^{b+1}.
\end{equation}
According to the  the estimates   \eqref{id2}  and  the identity \eqref{id1}, we get
  for all  $s \geq \max  (-\log T, S_4),$
\begin{equation}\label{sigma111b}
\Sigma_{2}^{5}(s)\le
\frac{C_2}{s^{a+1}}\ibint |w|^{p+1}\log^a(2+\p^2w^2)\w \y+  C e^{-s}.
\end{equation}
Hence, using  \eqref{1nov1}, \eqref{jan211}  and \eqref{sigma111b}, yields  for all  $s \geq \max  (-\log T, S_4),$
\begin{align}\label{1jan1}
\ibint  |\nabla w|^2\psi^2\w\y \le
&-\frac{2}{(p+1)s^{a}}\Big(1-\frac{(p+1)C_2}{2s}\Big)\ibint |w|^{p+1}\log^a(2+\p^2 w^2)  
 \psi^2\w\y\no\\
 &+ \frac4{p-1}\ibint  w\partial_sw \psi^2\w\y+K_{17}s^{b+1}+Ce^{-s}.
\end{align}
Taking $S_5\ge S_4$ large enough such that $1-\frac{(p+1)C_2}{2S_5}>0$, we have for all  $s \geq \max  (-\log T, S_5),$
\begin{equation*}%\label{1nov3}
\ibint  |\nabla w|^2 \psi^2\w\y \le \frac4{p-1}\ibint   w\partial_sw\ \psi^2w\y+K_{17}(R)s^{b+1}+Ce^{-s}.
\end{equation*}
Thus, \eqref{equ:boundLp123} follows from the property of $\psi$. This ends the proof  of Lemma \ref{lemm:tmpq2}.
\end{pf}
Now, we are ready to give the proof of Proposition    \ref{19dec3bis1}.\\

\medskip

  {\it{Proof of Proposition    \ref{19dec3bis1}:}}
[\textbf{Proof of \eqref{equ:keyintW} for all $q \geq 2 $ by a bootstrap argument}] \\
The proof is obtained by following the same 
part  in \cite{GMSiumj04}. However,   as explained before (see Remarks \ref{r1}),  in our case we  have two  additional problems. 
Let $R>0$ and  suppose that we have 
 \begin{equation}\label{equ:keyintWbis}
\int_s^{s+1} \|\grad w(\tau)\|_{L^2(\mathbf{B}_{4R})}^{2q} \t \leq K_{15}(q,4R,\vv )s^{\mu_2(q,4R,\vv)}, \qquad \forall s \geq
 \max  (-\log T, S_3), 
\end{equation}
for some $\mu_2(q,4R,\vv)>0$
and  for some $q \geq 2$.% let us show that \eqref{equ:keyintW} holds for all  $\tilde{q} \in [q, q +  \frac2{p-1} ].$

\medskip

Combining   \eqref{equ:keyintWbis}  and \eqref{31oct123}, we write for all $s \geq \max  (-\log T, \tilde{S_3})$, 
\begin{equation}\label{31oct123bis2}
\int_s^{s+1}  \|w(\tau )\|_{L^{\pp +1}(\mathbf{B}_{2R})}^{q(\pp+1)} d\tau  \leq  K_{18}(q,R,\vv )s^{\mu_3(q,R,\vv)}, \ \ \forall  \vv \in (0,p-1),
%  \ \ \forall s \geq \max  (-\log T, S_5).
\end{equation}
for some $\mu_3(q,R,\vv)>\mu_2(q,R,\vv)$.
where $\tilde{S_3}=\max (S_3,S_5)$. 
Thus,
%For this polynomial estimate, 
we  use   \eqref{18jan1},  \eqref{31oct123bis2} and apply Lemma  \ref{lemm:intpola} with $\alpha = q(\pp+1)$, $\beta =\pp+1$, $\gamma = \delta  = 2$ to get that for all
 $s \geq \max  (-\log T, \tilde{S_3})$, 
\begin{equation}\label{16jan1}
 \|w(s)\|_{L^\lambda(\mathbf{B}_{2R})} \leq K_{19}(q,R,\vv ) s^{\mu_4(q,R,\vv )}, \ \   \forall \lambda <  \pp+1 - \frac{\pp-1}{q + 1}.\ \   \forall \vv \in (0,p-1),
\end{equation}
for some $\mu_4(q,R,\vv)>\mu_3(q,R,\vv)$.
Thanks to the 
 Holder's inequality,  
\begin{equation}\label{15jan5}
\|\psi^2 w \partial_s w \|_{L^1(\mathbf{B}_{2R})} \leq \|\psi w \|_{L^\lambda (\mathbf{B}_{2R})} \times \|\psi \partial_s w \|_{L^{\lambda'}(\mathbf{B}_{2R})}, \quad \frac{1}{\lambda} + \frac{1}{\lambda'} = 1,
\end{equation}
 with  Lemma \ref{lemm:tmpq2}, \eqref{15jan5} and \eqref{16jan1}, we have for all $s \geq \max  (-\log T, \tilde{S_3})$, 
\begin{equation}\label{16jan11}
\|\grad w\|_{L^2(\mathbf{B}_R)}^{2} \leq K_{20}(q,r,\vv)s^{\mu_5(q,R,\vv )}\|\psi  \partial_s w \|_{L^{\lambda'}(\mathbf{B}_{2R})}+K_{20}(q,r,\vv)s^{b+1}.
\end{equation}
From now, we take $\lambda > 2$ and we choose $\varepsilon \in (0,\varepsilon_0]$ small enough.
Observe that $\lambda'>  \frac{p+1}{p}$ since $\lambda<p+1$. 
Let us now bound $\|\psi \partial_s w \|_{L^{\lambda'}(\mathbf{B}_{2R})}$. 
By using  Holder's inequality, we have
\begin{equation}\label{16jan2}
\|\psi \partial_s w\|_{L^{\lambda'}
(B_{2R})
} \leq \|\psi \partial_s w\|_{L^2(B_{2R})}^{1 - \theta} \times \|\psi \partial_s w\|_{L^{p_1-\varepsilon}(B_{2R})}^\theta, \quad  \frac{1}{\lambda'} = \frac{1 - \theta}{2} + \frac{\theta}{p_1-\varepsilon},
\end{equation}
%From now on, we take $\lambda > 2$, choose $\varepsilon_1$ small enough  and fix
$p_1=\frac{p+1}{p}$ and where 
 \begin{equation}\label{teta}
\theta = \frac{(\lambda - 2)(p+1-\varepsilon p)}{\lambda(p-1+\varepsilon p)} \in (0,1).
\end{equation}
Putting \eqref{16jan11} and \eqref{16jan2} together, we get  for all $s \geq \max  (-\log T, \tilde{S_3})$,
\begin{align}\label{16jan44}
\|\grad w\|_{L^2(\mathbf{B}_R)}^{2} \leq &K_{20}(q,R,\vv)s^{\mu_5(q,R,\vv)}\|\psi \partial_s w\|_{L^2(\mathbf{B}_{2R})}^{1 - \theta} \times \|\psi \partial_s w\|_{L^{p_1-\varepsilon}(\mathbf{B}_{2R})}^\theta\no \\
%\|\psi  \partial_s w \|_{L^{\lambda'}(\mathbf{B}_{2R})}
&+K_{20}(q,R,\vv)s^{b+1}.
\end{align}
By integrating inequality \eqref{16jan44} between $s$ and $s+1$, %and taking into account ,
 we obtain  for all $s \geq \max  (-\log T, \tilde{S_3})$, 
\begin{align}\label{16jan44i}
\int_s^{s+1} \|\grad w(\tau )\|_{L^2(\mathbf{B}_R)}^{2\tilde{q}} \t \leq & K_{21}(q,R,\vv) s^{\mu_6(q,R,\vv)\tilde{q}} \underbrace{\int_s^{s+1} \|\psi \partial_s w\|_{L^2(\mathbf{B}_{2R})}^{\tilde{q}(1 - \theta)} \times \|\psi \partial_s w\|_{L^{p_1-\varepsilon}(\mathbf{B}_{2R})}^{\tilde{q}\theta} \t}_{\Gamma (s)}\nonumber\\
&+K_{21}(q,R,\vv )s^{(b+1)\tilde{q}},
\end{align}
for some $\tilde{q} > q$.
Let $\alpha = \frac{2}{(1 - \theta)\tilde{q}}$ and use Holder's inequality in time, we obtain for all $s \geq \max  (-\log T, \tilde{S_3})$,
\begin{equation}\label{16jan7}
\Gamma (s) \leq \left(\int_s^{s+1}\|\psi  \partial_s w\|_{L^2(\mathbf{B}_{2R})}^2 \t\right)^{\frac{1}{\alpha}}
\left(\int_s^{s+1}\|\psi  \partial_s w\|_{L^{p_1-\varepsilon}(\mathbf{B}_{2R})}^{\tilde{q}\theta \alpha'} \t\right)^{\frac{1}{\alpha'}}, \quad \frac{1}{\alpha} + \frac{1}{\alpha'} = 1.
\end{equation}
%where we used $(i)$ in Proposition \ref{prop:boundEpsi}.\\
From the inequalities \eqref{18jan1}, \eqref{16jan44i} and \eqref{16jan7}, we infer that  for all $s \geq \max  (-\log T, \tilde{S_3}
)$,
\begin{align}\label{16jan44ii}
\int_s^{s+1} \|\grad w(\tau )\|_{L^2(\mathbf{B}_R)}^{2\tilde{q}} \t \leq & K_{22} (q,R,\vv )
s^{\mu_6(q,R,\vv)\tilde{q}}  \left(\int_s^{s+1}\|\psi  \partial_s w\|_{L^{p_1-\varepsilon}(\mathbf{B}_{2R})}^{\tilde{q}\theta \alpha'}\t\right)^{\frac{1}{\alpha'}}\no\\
&+K_{22}(q,R,\vv )s^{(b+1)\tilde{q}}.
\end{align}
%In order to estimate the right-hand side of inequality \eqref{16jan44ii},
Equipped with the arguments presented in the proof of   Lemmas 6.5 and 6.6 in \cite{GMSiumj04} and by exploiting 
Corollary \ref{19dec3}, 
it is straightforward 
to get, for all $s \geq \max  (-\log T, \tilde{S_3})$,
\begin{align}\label{17jan1}
\int_s^{s+1}\|\psi w_s\|_{L^{p_1-\varepsilon}(\mathbf{B}_{2R})}^{\tilde{q}\theta \alpha'}\t \leq&   K_{23} (q,R,\vv ) \int_s^{s+1}\left\|{\frac1{\tau^a}|w|^p}{\log^a(2+\psi^2w^2)}\right\|_{L^{p_1-\varepsilon}(\mathbf{B}_{2R})}^{\tilde{q}\theta \alpha'}\t\no\\
&+ K_{23} (q,R,\vv )s^{b+1}.
\end{align}
%Since
%\begin{equation}\label{25jan1}
%\left\|{\frac1{\tau^a}|w|^p}{\log^a(2+\psi^2w^2)}\right\|_{L^{p_1-\vv}(\mathbf{B}_{2R})}^{p_1-\varepsilon_1}=\int_{\mathbf{B}_{2R}}\frac1{\tau^{a(1+\frac1{p}-\varepsilon_1)}} |w|^{p+1-\varepsilon_1 p}\log^{a(1+\frac1{p}-\varepsilon_1)}(2+\p^2 w^2)  
% \y.
%\end{equation}
%Clearly
%\begin{equation}\label{25jan2}
%\int_{\mathbf{B}_{2R}}\frac1{\tau^{a(1+\frac1{p}-\varepsilon_1)}} |w|^{p+1-\varepsilon_1 p}\log^{a(1+\frac1{p}-\varepsilon_1)}(2+\p^2 w^2) \le C +C\int_{\mathbf{B}_{2R}}\frac1{\tau^{a}} |w|^{p+1}\log^{a}(2+\p^2 w^2) 
% \y.
%\end{equation}
By combining \eqref{17jan1},  \eqref{equiv44} and the identity $
e^{-\frac{ps}{p-1}}s^{\frac{a}{p-1}}   |f(\p w)|=\frac1{s^{a}}|w|^{p}\log^a(2+\p^2w^2),$
 we deduce that for all $s \geq \max  (-\log T, \tilde{S_3})$,
\begin{equation}\label{6fev1}
\int_s^{s+1}\|\psi w_s\|_{L^{p_1-\varepsilon}(\mathbf{B}_{2R})}^{\tilde{q}\theta \alpha'}\t \leq   K_{24} (q,R,\vv ) \int_s^{s+1}\left\||w|^{p+\tilde\varepsilon}\right\|_{L^{p_1-\varepsilon}(\mathbf{B}_{2R})}^{\tilde{q}\theta \alpha'}\t+ K_{24} (q,R,\vv )s^{b+1},
\end{equation}
where $\tilde\varepsilon=\frac{p(p-1)\varepsilon}{p+1-\varepsilon p}.$ Therefore, 
\begin{equation}\label{30jan}
\int_s^{s+1}\|\psi w_s\|_{L^{p_1-\varepsilon}(\mathbf{B}_{2R})}^{\tilde{q}\theta \alpha'}\t 
 \le  K_{24} (q,R,\vv ) \int_s^{s+1}\Big(\int_{\mathbf{B}_{2R}} |w|^{p+1-\varepsilon}
 \y\Big)^{\frac{p\tilde{q}\theta \alpha'}{p+1-\varepsilon p}}\t+ K_{24} (q,R,\vv ) s^{b+1}.
\end{equation}
Using together  \eqref{31oct123} and \eqref{30jan}, we obtain
\begin{align*}
\int_s^{s+1}\|\psi w_s\|_{L^{p_1-\varepsilon}(\mathbf{B}_{2R})}^{\tilde{q}\theta \alpha'}\t 
& \le 
    K_{25} (q,R,\vv ) \int_s^{s+1}\left\|\grad w\right\|_{L^2(\mathbf{B}_{4R})}^{\frac{2p\tilde{q}\theta \alpha'}{p+1-\varepsilon p}}d\tau+  K_{25} (q,R,\vv )(s^{\frac{(b+1)p\tilde{q}\theta \alpha'}{p+1-\varepsilon p}}+ s^{b+1}).
\end{align*}
%$\theta = \frac{(\lambda - 2)(p+1-\varepsilon_1 p)}{\lambda(p-1-\varepsilon_1 p)} \in (0,1)$. 
By Proposition 6.4 in \cite{GMSiumj04}, we have
 $
\frac{2p\tilde{q}\theta \alpha' }{p+1-\varepsilon_1p}
%= \frac{2p\tilde{q}(\lambda - 2)(p+1-\varepsilon_1 p)}{\lambda(p-1+\varepsilon_1 p)(2-(1-\theta)\tilde{q})}
< 2q$, (for $\varepsilon_1$ small enough) for all $\tilde{q} \in [q, q + \frac{2}{p+1}]$. Then, by using the  inequality, for all
 $r\in [1, 2q]$,  $X^{r}\le C +CX^{2q}$, for all $X>0$, we write for all $s \geq \max  (-\log T, \tilde{S_3})$, for all $\varepsilon \in (0,\varepsilon_1],$
\begin{equation}\label{25jan4}
\int_s^{s+1}\|\psi w_s\|_{L^{p_1-\varepsilon}(\mathbf{B}_{2R})}^{\tilde{q}\theta \alpha'}\t 
 \le   K_{26} (q,R,\vv ) \int_s^{s+1}\left\|\grad w\right\|_{L^2(\mathbf{B}_{4R})}^{2q} \t +  K_{26} (q,R,\vv ) s^{2q(b+1)}.
\end{equation}
From \eqref{16jan44ii} and \eqref{25jan4}, we have for all $s \geq \max  (-\log T, \tilde{S_3})$,
\begin{align}\label{25jan5}
\int_s^{s+1} \|\grad w(\tau )\|_{L^2(\mathbf{B}_R)}^{2\tilde{q}} \t \leq    K_{27} (q,R,\vv )  s^{\mu_7(q,R,\vv)\tilde{q}} \left(
  \int_s^{s+1}\left\|\grad w\right\|_{L^2(\mathbf{B}_{4R})}^{2q}d\tau\right)^{\frac{1}{\alpha'}}\nonumber\\
+  K_{27} (q,R,\vv ) s^{\mu_8(R,\vv ,q ,\tilde{q})}.
\end{align}
%\begin{align}\label{116jan44iii}
%\int_s^{s+1} \|\grad w(\tau )\|_{L^2(\mathbf{B}_R)}^{2\tilde{q}} \t \leq&  K_9 s^{(b+1)(q\tilde{q}+\frac{1}{\alpha})}\left(
%  \int_s^{s+1}\left\|\grad w\right\|_{L^2(\mathbf{B}_{4R})}^{2q}d\tau\right)^{\frac{1}{\alpha'}}\\
%&+K_9s^{
%{(b+1)(q\tilde{q}+\frac{1}{\alpha})}+
%\frac{(b+1)p\tilde{q}\theta }{p+1-\varepsilon_1}} 
%%{\frac{1}{\alpha'}}
%+K_9s^{{(b+1)(q\tilde{q}+\frac{1}{\alpha})}+(b+1)\tilde{q}}.
%\end{align}
Therefore, estimates \eqref{equ:keyintWbis} and \eqref{25jan5} lead to the following:
\begin{align}\label{116jan44iii1}
\int_s^{s+1} \|\grad w(\tau )\|_{L^2(\mathbf{B}_R)}^{2\tilde{q}} \t \leq   K_{28} (q,R,\vv )    s^{\mu_9(q,R,\vv,q ,\tilde{q})}.
&%+K_9s^{{(b+1)(q\tilde{q}+1)}+\frac{(b+1)p\tilde{q}\theta }{p+1-\varepsilon_1}} 
%{\frac{1}{\alpha'}}
%+K_9s^{{\tilde{q} (b+1)(q+2)}}.
\end{align}
Thus, inequality \eqref{equ:keyintW} is valid for all $\tilde{q} \in [q, q + \frac{2}{p+1}]$. Repeating this argument, we would obtain that \eqref{equ:keyintW} holds for all $q \geq 2$. This concludes the proof of  Proposition \ref{19dec3bis1}.
\Box

\subsection{A polynomial bound for the $H^{1}(\mathbf{B}_R)$ norm of   solution of equation  \eqref{A}} \label{subssection2}
Based on Proposition \ref{19dec3bis1},    
we  are in position to  derive  
a polynomial bound for the  $H^{1}(\mathbf{B}_R)$  norm. 
  More precisely, the aim of this subsection is to prove Proposition \ref{prop21},    

\medskip

\textit{Proof of Proposition}.
\ref{prop21}\

First, we  use  \eqref{18jan1}, Proposition \ref{19dec3bis1} and apply Lemma  \ref{lemm:intpola} with $\alpha = q(p- \frac{ \varepsilon}2+1)$, $\beta =p- \frac{ \varepsilon}2+1$, $\gamma = \delta  = 2$ to get that, for all $s \geq \max  (-\log T, \tilde{S_3})$, 
\begin{equation}\label{equ:tmp2key11}
%\sup_{\tau \in [s,s+1]}
 \|w(s)\|_{L^\lambda(\mathbf{B}_R)} \leq K_{29}(q,R,\vv) s^{\mu_{10}(q,R,\vv)}, \ \   \forall \lambda <  p- \frac{ \varepsilon}2+1 - \frac{
p- \frac{ \varepsilon}2
-1}{q + 1},\ \   \forall  \vv \in (0,\varepsilon_1],
\end{equation}
where $\tilde{S_3}=\max (S_3,S_5)$. 
%Hence  the estimate \eqref{equ:tmp2key}, is true for all  $\forall \lambda <  \pp+1 - \frac{\pp-1}{q + 1}, \forall \pp<p, \forall q\ge 2$. then 
%\begin{equation}\label{equ:tmp2keyy}
%\sup_{s \geq\hat{s}_2} \|w(s)\|_{L^\lambda(\mathbf{B}_R)} \leq C_2(R,K_q), \quad  \forall \lambda <  \lambda_1=p+1.
%\end{equation}
Clearly, there exists $\varepsilon_2=\varepsilon_2(p,N,q)>0$ such that, for all $\varepsilon \in (0,\varepsilon_2]$, we have 
 $q=\frac{2p- \varepsilon}{ \varepsilon}-1\ge 2$. Therefore, 
for all $\varepsilon \in (0,,\varepsilon_2],$ for all   
 $s \geq \max  (-\log T,\tilde S_3)$
 we have 
\begin{equation}\label{s1jan19}
\int_{\mathbf{B}_R}\!|w(y,s )|^{
p+1- \varepsilon
}{\mathrm{d}}y
\le K_{30}(\varepsilon,R) s^{\mu_{11}(R,\varepsilon)}.
\end{equation}

We are now ready to    Control of  $\grad w$ in $L^{2}({\mathbf{B}_R)}$.  In fact, we  use the  Gagliardo-Nirenberg inequality in order to claim  the following:
\begin{lem}\label{lem24}  There exists $\varepsilon_3=\varepsilon_3(p,N)\in (0,\varepsilon_2]$ such that, for all $\varepsilon \in (0,\varepsilon_3]$, for all   
   $s \geq \max  (-\log T, \tilde{S_3})$   we have 
\begin{equation}\label{s40}
\iint\!\psi^2| w(y,s )|^{p+1+\varepsilon}{\mathrm{d}}y
\le  K_{31} (R,\varepsilon)s^{\mu_{12}(R,\varepsilon)} \Big(\iint\psi^2|\grad w(y,s )|^{2}{\mathrm{d}}y\Big)^{\beta}+K_{31}(R,\varepsilon) s^{\mu_{12}(R,\varepsilon)},
\end{equation}
where $\beta=\beta(p,N,\varepsilon) \in (0,1)$ and $\mu_{12}=\mu_{12}(R,\varepsilon)>0.$
\end{lem}
\begin{pf}

Let $\varepsilon
\in (0,\varepsilon_2)$.   By interpolation, we write 
\begin{equation}\label{s390f}
\iint\! \psi^2| w(y,s )|^{p+1+\varepsilon}{\mathrm{d}}y
\le  \Big(\iint\! \psi^{\nu} |w(y,s)|^{p+1-\varepsilon}{\mathrm{d}}y\Big)^{\eta} \Big(\iint
|\psi w(y,s )|^{r}{\mathrm{d}}y\Big)^{1-\eta},
\end{equation}
where %$\varepsilon$ small enough   and where
$$\nu=2\frac{
r(1-\varepsilon)-(p+1+\varepsilon)}{
r-(p+1+\varepsilon)
}, \quad \eta=\frac{r-(p+1+\varepsilon)}{r-(p+1-\varepsilon)}, \ \ 
%\beta_0=\frac{2\varepsilon}{2^*-(p+1-\varepsilon)} 
$$
where
\begin{equation}\label{r4}
r=\left\{
\begin{array}{l}
\frac{2N}{N-2},\qquad  if N\ge 3,\\
\\
p+2,\qquad  if N= 2,
\end{array}
\right.
\end{equation}
and where $\varepsilon <r-p-1$.
% \ \ \beta=\frac{{p-1+3\varepsilon}}{2q-(p+3-\varepsilon)}
Exploiting  the fact 
that there exists $\tilde\varepsilon_2=\tilde\varepsilon_2(p,N)\in (0,r-p-1)$  small enough such that for $\varepsilon \in (0,\tilde\varepsilon_2],$ we have  $\nu=\nu(p,\varepsilon) \in [1,2).$
Therefore, by using the properties of $\psi$ given by
\eqref{equ:psiy} and  the estimate \eqref{s1jan19}  
  we get
\begin{equation}\label{01pf}
 \iint\! \psi^{\mu} |w(y,s)|^{p+1-\varepsilon}{\mathrm{d}}y\le
\int_{\mathbf{B}_{2R}}\!|w(y,s )|^{
p+1- \varepsilon
}{\mathrm{d}}y
\le K_{30}(\varepsilon,2R) s^{\mu_{11}(\vv, 2R)}.
\end{equation}
Thanks to  \eqref{s390f}, \eqref{01pf} and the Sobolev embedding, we conclude
\begin{equation}\label{s39F}
\iint \psi^2| w(y,s )|^{p+1+\varepsilon}{\mathrm{d}}y
\le  K_{32}(R,\varepsilon) s^{\mu_{13}(\vv,R)} \Big(\iint\!|\grad \big(\psi w(y,s )\big)|^{2}{\mathrm{d}}y\Big)^{\beta},
\end{equation}
where
$$
\beta=\frac{r\varepsilon}{r-(p+1-\varepsilon)}.
$$
Note that, by exploiting the inequality $
|\grad \big(\psi w\big)|^2\le 
2 \psi^2|\grad w|^2 +
 2|\grad \psi|^2 w^2$, the properties of $\psi$ given by \eqref{equ:psiy}  and 
the fact that $\|\nabla \psi\|_{L^\infty}\le C$,  we obtain 
\begin{equation}\label{s3900F}
 \iint\!|\grad \big(\psi w(y,s )\big)|^{2}{\mathrm{d}}y\le
C \iint\!\psi^2|\grad w(y,s )|^{2}{\mathrm{d}}y +C
\int_{\mathbf{B}_{2R}}
  w^2(y,s ){\mathrm{d}}y.
\end{equation}
From \eqref{s39F},  \eqref{s3900F}and  \eqref{18jjjjj}, we conclude 
\begin{equation}\label{00s39}
\iint \psi^2| w(y,s )|^{p+1+\varepsilon}{\mathrm{d}}y
\le  K_{33}(\varepsilon,R) s^{\mu_{14}(\vv, R)} \Big(\iint\!\psi^2|\grad w(y,s )|^{2}{\mathrm{d}}y\Big)^{\beta}+ K_{33}(\varepsilon,R) s^{\mu_{14}(\vv, R)}.
\end{equation}
  %$\grad \big(\psi w(y,s )\big)= \grad \psi w(y,s )+ \grad \big(\psi w(y,s )\big)$
Now, if  $\varepsilon_3\le \tilde \varepsilon_2$  is chosen   small enough such that     $\beta=\frac{r\varepsilon_3}{r-(p+1-\varepsilon_3)}\in (0,1)$, then the estimate   \eqref{00s39} implies 
 \eqref{s40}.%  follows  immediately from  \eqref{00s39},
 This ends the proof  of Lemma %which ends the proof of Lemma
  \ref{lem24}.
\end{pf}

\textit  {Proof of Proposition 
\ref{prop21}:
From  \eqref{Eloc1}, the definition
 \eqref{Eloc}
of the  local functional:
  $\mathcal{E}_\psi(w(s),s)$, we see that for all  $s \geq \max  (-\log T, S_4),$
\begin{align}\label{19fev1bis}
\iint \psi^2|\grad w|^2\w \y  &\le 2\iint
e^{-\frac{(p+1)s}{p-1}}s^{\frac{2a}{p-1}}  \psi^2  F(\p w)\w \y + 2K_{13}(R)s^{b+1}.
\end{align}
Thanks  to  \eqref{equiv4} and  \eqref{19fev1bis} and the fact that $\rho (2R)\le \w \le 1,$ for all $y\in  B_{2R} $,
we conclude for all   $s \geq \max  (-\log T, S_4)$
\begin{align}\label{19fev1}
\iint \psi^2|\grad w|^2\y &\le K_{34}(R,\varepsilon)\iint \psi^2 |w(y,s )|^{p+ \varepsilon+1} \y + K_{34}(R,\varepsilon)s^{b+1}.
\end{align}
According to  \eqref{19fev1} together with  Lemma \ref{lem24} in the particular case when $\varepsilon=\varepsilon_3$, we have for all $s \geq \max  (-\log T, \tilde{S_3})$
\begin{equation}\label{s40b}
\iint \psi^2|\grad w|^2\y
\le  K_{35} (R,\varepsilon_3)s^{\mu_{12}(R,\varepsilon_3)} \Big(\iint\psi^2|\grad w|^{2}{\mathrm{d}}y\Big)^{\beta}+K_{35}(R,\varepsilon_3)s^{\mu_{12}(R,\varepsilon_3)},
\end{equation}
where $\beta=\beta(p,N,\varepsilon_3) \in (0,1).$
It suffices to combine \eqref{s40b}  and the fact that $\beta<1$, to obtain that  for all $s \geq \max  (-\log T, \tilde{S_3})$
\begin{equation}
\iint \psi^2|\grad w|^2\y
\le K_{36}(R,\varepsilon_3)s^{\frac{\mu_{12}(R,\varepsilon_3)}{1-\beta}}.\label{second222}
\end{equation}
Clearly, by combining \eqref{second222},  \eqref{18jjjjj} and \eqref{equ:psiy}, we  conclude \eqref{feb19}, 
where $\mu=\frac{\mu_{12}(R,\varepsilon_3)}{2-2\beta}$, which yields the conclusion of Proposition \ref{prop21}}. \Box

%%%%%%%%%%%%%%%%%%%%%%%%%%%%%%%%%%%%%%%%%%%%%%%%%%%%%%%%%%%%%%%%%%%%%%%%
%%%%%%%%%%%%%%%%

\section{Proof of Theorem \ref{t1} and Theorem \ref{t2}}\label{section3}
In this section,  thanks to   polynomial  estimate obtained in  Proposition \ref{prop21}, we  prove 
Theorem \ref{t1} and Theorem \ref{t2} here. 
This section is divided into two parts:
\begin{itemize}
\item  In subsection \ref{3.2},    we prove  Theorem \ref{t1}. More precisely,  based upon Proposition  \ref{prop21}, we construct  a Lyapunov functional 
for equation (\ref{A}) and a blow-up criterion involving this functional.
\item In subsection \ref{3.3}, we prove   Theorem \ref{t2}.
\end{itemize}

%\end{document}
%%%%%%%%%%%%%%%%%%%%%%%%%%%%%%%%%%%%%%%%%%%%%%%%%%%%%%%%%%%%%%%%%%%%%%%%%%%%%%%%%%%%%%%%%%%%%%%%%%%%%%%%%%%%%%%%%%%%%
%%%%%%%%%%%%%%%%%%%%%%%%%%%%%%%%%%%%%%%%%%%%%%%%%%%%%%%%%%%%%%%%%%%%%%%%%%%%%%%%%%%%%%%%%%%%%%%%%%%%%%%%%%%%%%%%%%%%%
%%%%%%%%%%%%%%%%%%%%%%%%%%%%%%%%%%%%%%%%%%%%%%%%%%%%%%%%%%%%%%%%%%%%%%%%%%%%%%%%%%%%%%%%%%%%%%%%%%%%%%%%%%%%%%%%%%%%%
%%%%%%%%%%%%%%%%%%%%%%%%%%%%%%%%%%%%%%%%%%%%%%%%%%%%%%%%%%%%%%%%%%%%%%%%%%%%%%%%%%%%%%%%%%%%%%%%%%%%%%%%%%%%%%%%%%%%%
%%%%%%%%%%%%%%%%%%%%%%%%%%%%%%%%%%%%%%%%%%%%%%%%%%%%%%%%%%%%%%%%%%%%%%%%%%%%%%%%%%%%%%%%%%%%%%%%%%%%%%%%%%%%%%%%%%%%%
%%%%%%%%%%%%%%%%%%%%%%%%%%%%%%%%%%%%%%%%%%%%%%%%%%%%%%%%%%%%%%%%%%%%%%%%%%%%%%%%%%%%%%%%%%%%%%%%%%%%%%%%%%%%%%%%%%%%%
%%%%%%%%%%%%%%%%%%%%%%%%%%%%%%%%%%%%%%%%%%%%%%%%%%%%%%%%%%%%%%%%%%%%%%%%%%%%%%%%%%%%%%%%%%%%%%%%%%%%%%%%%%%%%%%%%%%%%
%%%%%%%%%%%%%%%%%%%%%%%%%%%%%%%%%%%%%%%%%%%%%%%%%%%%%%%%%%%%%%%%%%%%%%%%%%%%%%%%%%%%%%%%%%%%%%%%%%%%%%%%%%%%%%%%%%%%%
%%%%%%%%%%%%%%%%%%%%%%%%%%%%%%%%%%%%%%%%%%%%%%%%%%%%%%%%%%%%%%%%%%%%%%%%%%%%%%%%%%%%%%%%%%%%%%%%%%%%%%%%%%%%%%%%%%%%%
%%%%%%%%%%%%%%%%%%%%%%%%%%%%%%%%%%%%%%%%%%%%%%%%%%%%%%%%%%%%%%%%%%%%%%%%%%%%%%%%%%%%%%%%%%%%%%%%%%%%%%%%%%%%%%%%%%%%%
\subsection{A Lyapunov functional}\label{3.2}

In this subsection, our aim is to construct a Lyapunov functional for equation \eqref{A}.
 Note that this functional is far from being trivial and makes our main contribution.
More precisely, thanks to the  rough estimate obtained in  the Proposition \ref{prop21}, 
we derive here that the functional  $L(w(s),s)$ defined in \eqref{10dec2} is a decreasing 
  functional  of time  for equation (\ref{A}),  provided that is   $s$ large enough. \\

\medskip

Let us remark that in   Section \ref{section2}, we construct a Lyapunov functional $N_{m_0}(w(s),s)$ defined in \eqref{lyap1}, but we obtain just a rough estimate because the multiplier is not bounded. Nevertheless, the multiplier  related to  the functional  $L(w(s),s)$
 is  bounded. Then,  as we said above,   the natural energy
  $E(w(s),s)$ defined in \eqref{18jan1} is  a small perturbation of  $L(w(s),s)$.
%\end{nb}

\medskip

In order to prove  that the  functional $L(w(s),s)$ is a Lyapunov functional, we start by using 
the  additional information obtained in Section \ref{section2}, to write  several useful lemmas 
 which play key roles in our analysis. More precisely, we start  by  stating the following:
\begin{lem}  \label{lemmain}  For all  $r\in [2,2^*)$,   for all  
$s\geq  \widehat{s}_1=\max (-\log T,\widehat{S}_1)$, we have
\begin{align}\label{claim01}
\iint  |  w(y,s)|^{r}\w \y\le
 M_1{s^{\sigma r}},
\end{align}
where  $\sigma=\mu (a,p,N,\frac12)$, $M_1$ depends on $ p, a, N,r$  and
$\|w(  \widehat{s}_1)\|_{H^1}$ and where $2^*=\frac{2N}{N-2}$, if $N\ge 3$ and $2^*=\infty$,  if $N=2$.
\end{lem}
Throughout the proof we employ the following notations:

\medskip

The ball in $\er^N$ with radius   $R$ around the point $z$ is denoted 
 ${\mathbf{D}(z,R)}=\{ x\in \R^N, \|x-z\|_{\infty}\le R
\}$, where the infinity norm 
is  given by the formula
$\ds{\|x\|_{\infty}=\sup_{1\le i\le N}|x_i|}.$
Also, 
the ball in $\er^N$ with radius   $R$ around the point $z$ is denoted 
 ${\mathbf{B}(z,R)}=\{ x\in \R^N, |x-z|\le R
\}$, where the  norm 
is  given by 
$\ds{|x|=\sqrt{\sum_{i=1}^Nx_i^2}}.$
Finally,
let us recall that   theses norms on $\er^N$ are equivalent.  In fact,  we have 
\begin{equation}\label{nequi}
\|x\|_{\infty}\le |x|\le \sqrt{N}\|x\|_{\infty}, \qquad \forall x\in \er^N.
\end{equation}

\begin{pf}  In order to obtain the estimate \eqref{claim01},  
we  combine  a covering technique and the result obtained in   Proposition \ref{prop21}.

\medskip
First, we claim that  $\R^N=\cup_{z\in Z^N} {\mathbf{D}(z,\frac12)}$ and  
the sequence  $ \Big({\mathbf{D}(z,\frac12)}\Big)_{z\in Z^N}$ 
are arbitrary pairwise  sets are negligible.
Let  $r\in [2,2^*]$.
As an immediate consequence, we write
\begin{align}\label{claim0010}
\iint  |  w_{x_0}(y,s)|^{r}\w\y=&\sum_
{z\in Z^N}
\int_{\mathbf{D}(z,\frac12)}  |  w_{x_0}(y,s)|^{r}\w \y \nonumber\\
\le &
\sum_{z\in Z^N} 
\Big(\sup_{y\in\mathbf{D}(z,\frac12)} \w\Big)
\int_{\mathbf{D}(z.\frac12)} 
 |  w_{x_0}(y,s)|^{r} \y .
\end{align}
 Note that using the definition \eqref{scaling} of $w_{x_0},$ we  see that 
\begin{equation}\label{def00}
\textrm {for all}\ \  y,z \in \R^N, w_{x_0}(y+z,s)=w_{x_0+ze^{-s/2}}(y,s)
\end{equation}
From \eqref{nequi} and \eqref{def00}, for all $ z \in \R^N$,   $s\geq  \widehat{s}_1=\max (-\log T,,\widehat{S}_1)$
\begin{align}\label{claim010}
\int_{\mathbf{D}(z,\frac12)}  |  w_{x_0}(y,s)|^{r}\y &\le  
\int_{\mathbf{B}(z,\frac{\sqrt{N}}2)}  |  w_{x_0}(y,s)|^{r}\y=
\int_{\mathbf{B}(0,\frac{\sqrt{N}}2)}  |  w_{x_0}(y+z,s)|^{r} \y\no\\
&=
\int_{\mathbf{B}(0,\frac{\sqrt{N}}2)}  |  w_{x_0+ze^{-s/2}}(y,s)|^{r} \y.
%\le   K_{36}{s^{rq}}.
\end{align}

%Since, there exists $R>1$ such that for all $z\in \R^N$, we have
%$_D_z\subseteq\mathbf{B}(z,R)$, 
%$D(z,R)=\{ x=(x_i)_{1\le i\le N}\in \R^N, \sup_{1\le i \le N} |x_i-z_i|\le R
%\}$, where $x=(x_i)_{1\le i\le N}\in \R^N$.

%for all $z_1,z_2 \in Z^N$, we have 
%${\mathbf{D}(z_1,\frac12)}\cap {\mathbf{D}(z_2,\frac12)}=\emptyset$.  
%To verify the last estimate we note that for the complex-valued functions u and v with α > 0
%there are two inequalities,
%Applying Proposition \eqref{polybound}, 
Thanks to
 \eqref{16dec5bis}
%\eqref{polybound}
 and
 \eqref{claim010}, we have   for all $ z \in \R^N$,   $s\geq  \widehat{s}_1=\max (-\log T,\widehat{S}_1)$
  %where $x=(x_i)_{1\le i\le N}\in \R^N$.
\begin{align}\label{claim010bis}
\int_{\mathbf{D}(z,\frac12)}  |  w_{x_0}(y,s)|^{r}\y \le   M_{2}{s^{r\sigma}},
\end{align}
where $\sigma=\mu (a,p,N,\frac12)$  and where  $M_2$ depends on $ p, a, N$  and
$\|w( \widehat{s}_1)\|_{H^1}$.
By exploiting \eqref{claim010bis} and \eqref{claim0010}, we have  for all $x_0, z \in \R^N$,  $s\geq  \widehat{s}_1=\max (-\log T,\widehat{S}_1)$
\begin{align}\label{claim0010bis}
\iint  |  w_{x_0}(y,s)|^{r}\w\y
\le & M_{2}{s^{r\mu}}
\sum_{z\in Z^N} 
\sup_{y\in\mathbf{D}(z,\frac12)} \w.
\end{align}
To complete the proof, it remains to
control   the right-hand side of  \eqref{claim0010bis}. More precisely, the term $\ds{\sum_{z\in Z^N} 
\sup_{y\in\mathbf{D}(z,\frac12)} \w}$.
Using the fact that  for all $z\in \R^N$, for all $y\in{\mathbf{D}(z,\frac12)}$, we have 
\begin{align}\label{15sep1}
\|z\|_{\infty}\le \|y\|_{\infty}+\|y-z\|_{\infty} \le  \|y\|_{\infty}+\frac12.
\end{align}
Therefore, by using  the   basic inequality $(a+b)^2\le 2 a^2+2b^2,$  for all $a,b >0,$  we set 
\begin{align}\label{15sep1b}
\|z\|_{\infty}^2\le  \big(\|y\|_{\infty}+\frac12\big)^2\le 2 \|y\|_{\infty}^2+\frac12.
\end{align}
In view of \eqref{15sep1b}, \eqref{nequi}, we have, for all $z\in \R^N$, for all $y\in{\mathbf{D}(z,\frac12)}$, we have 
\begin{align}\label{15sep2}
|y|^2\ge
\|y\|^2_{\infty} \ge \frac12\|z\|^2_{\infty}-\frac14\ge \frac1{2N}|z|^2-\frac14.
\end{align}
Due to  \eqref{15sep2} and to  the definition of $\rho$ given by \eqref{rho}, we conclude  for all $z\in \R^N$, 
\begin{align}\label{claim0010002}
\sup_{y\in\mathbf{D(}z,\frac12)} \w\le C e^{-\frac{|z|^2}{8N}}.
\end{align}
Thank to \eqref{claim0010002}, we get
\begin{align}\label{claim0010bis1}
\sum_{z\in Z^N} 
\sup_{y\in\mathbf{D}(z,\frac12)} \w\le C\sum_{z\in Z^N} 
e^{-\frac{|z|^2}{8N}}\le C\prod_{i=1}^{N} \sum_{z_i\in Z} 
e^{-\frac{z_i^2}{8N}}\le C.
\end{align}
By combining \eqref{claim0010bis1} and \eqref{claim0010bis},
we
easily obtain \eqref{claim01}. This concludes the proof  of Lemma    \ref{lemmain}. 
\end{pf}
Thanks to   of Lemma    \ref{lemmain}, we are in position to state the following:
\begin{lem}  \label{lemmain1}  For all   
$s\geq  \widehat{s}_1=\max (-\log T,\widehat{S}_1)$, we have 
\begin{align}\label{claim1}
%\frac1{s^a}
\iint  {|  w|^{p+1}}\log^{{a}}(2+\p^2 w^2  )\log (2+w^2 ) \w \y\le&
 M_{3}{s^{\frac14}}\iint  {|  w|^{p+1}}\log^{{a}}(2+\p^2 w^2  ) \w \y\nonumber\\
&+ M_{3} {s^{a+\frac14}},
\end{align}
where, $M_3$ depends on $ p, a, N$  and
$\|w(  \widehat{s}_1)\|_{H^1}.$
\end{lem}
\begin{nb} Let us mention that, in  the first term on the right-hand side the choice of the 
 power $\frac14$  is not optimal. In fact, with the same proof, one can show  the same estimate
with the  power   $\nu$, for any $\nu>0$, instead of the power $\frac14$. Let us denote that,  we can construct a Lyapunov functional, when we have the estimate above for some  power $\nu$  such that   $\nu\in (0,1 )$  instead of the power $\frac14$.
\end{nb}
\begin{pf}
 Let $\varepsilon \in (0,1)$. By
using  the inequality   $\log (2
+z^2)\le C+|z|^{\varepsilon^2}$,    for all $ z\in \R$, we conclude that
\begin{align}\label{22fev1}
\iint  \! {|  w|^{p+1}}\log^{{a}}(2+\p^2 w^2  )\log (2+w^2 ) \w \y\le \! C \!\iint  \! {|w|^{p+1}}\log^{{a}}(2+\p^2 w^2  ) \w \y\nonumber\\
+\iint  {|  w|^{p+1+\varepsilon^2}}\log^{{a}}(2+\p^2 w^2  ) \w \y.\qquad \qquad
\end{align}
Furthermore, we apply the interpolation in Lebesgue spaces to
get
%Thanks to interpolation in Sobolev spaces
\begin{align}\label{22fev2}
\iint  {|  w|^{p+1+\varepsilon^2}}\log^{{a}}(2+\p^2 w^2  ) \w \y\le\Big( \iint  {|  w|^{p+1}}\log^{{a}}(2+\p^2 w^2  ) \w \y\Big)^{1-\varepsilon}\nonumber\\
\Big( \iint  {|  w|^{p+1+\varepsilon}}\log^{{a}}(2+\p^2 w^2  ) \w \y\Big)^{\varepsilon}.
\end{align}
By combining\eqref{equiv1}, \eqref{equiv4} and   the inequality   $|z|^{p+1+\varepsilon}\le 1+|z|^{p+1+2\varepsilon}$,    
for all $ z\in \R$, we obtain
\begin{align}\label{22fev22}
\frac1{s^{ a}}\iint  {|  w|^{p+1+\varepsilon}}\log^{{a}}(2+\p^2 w^2  ) \w \y\le
 C
+C \iint  {|  w|^{p+1+2 \varepsilon}} \w \y.
\end{align}
Since $ p<p_S=\frac{N+2}{N-2}$, we then  choose $ \varepsilon_4$ small enough, such that for all $\varepsilon \in (0,\varepsilon_4 ]$  we have $p+1+2 \varepsilon < 2^* $
where $2^*=\frac{2N}{N-2}$, if $N\ge 3$ and $2^*=\infty$,  if $N=2$.
Therefore,  estimate   \eqref{claim01} implies that, for  all   
$s\geq  \widehat{s}_1=\max (-\log T,\widehat{S}_1)$,  for all  $\varepsilon\in [0,\varepsilon_4]$,
\begin{align}\label{22fev555}
\iint  {| w|^{p+1+2 \varepsilon}}  \w\y\le  \iint  {| w|^{p+1}}  \w\y+\iint  {| w|^{p+1+2 \varepsilon_4}}  \w\y\le
M_{4} s^{\sigma_3},
\end{align}
where  $\sigma_3$ depends on $ p, a, N,\varepsilon_4$  and $M_4$ depends on $ p, a, N,\varepsilon_4$  and
$\|w(  \widehat{s}_1)\|_{H^1}$.

By  combining \eqref{22fev2}, \eqref{22fev22} and   \eqref{22fev555}, we deduce that, for all  $s\geq  \widehat{s}_1=\max (-\log T,\widehat{S}_1)$,
  for all $ \varepsilon\in ( 0,\varepsilon_4]$.
\begin{align}\label{22fev222}
\iint  {|  w|^{p+1+\varepsilon^2}}\log^{{a}}(2+\p^2 w^2  ) \w \y\le M_{5}
 s^{(\sigma_3+a)\varepsilon}
\Big( \iint  {|  w|^{p+1}}\log^{{a}}(2+\p^2 w^2  ) \w \y\Big)^{1-\varepsilon}.
\end{align}
Thanks to  the basic inequality  $|a_1|^{\nu}|a_2|^{1-\nu}\le C|a_1|+C |a_2|$,  for all $ a_1,a_2\in \R$, for all $ \nu \in   ( 0, 1 ), $ we conclude that, for all 
  $s\geq  \widehat{s}_1=\max (-\log T,\widehat{S}_1)$,
  for all $ \varepsilon\in ( 0,\varepsilon_4]$,
\begin{align}\label{22fev222bis}
\iint  {|  w|^{p+1+\varepsilon^2}}\log^{{a}}(2+\p^2 w^2  ) \w \y
\le M_{6} s^{\sigma_3\varepsilon}
\Big(s^{ a}+ \iint  {|  w|^{p+1}}\log^{{a}}(2+\p^2 w^2  ) \w \y\Big).
\end{align}
Now, we choose  $ \varepsilon_5\in   ( 0,\varepsilon_4], $ such that $  \sigma_3\varepsilon_5 <\frac14.$
Then,  by \eqref{22fev1} and  \eqref{22fev222bis}, 
we
easily obtain \eqref{claim1}. This concludes the proof  of Lemma \ref{lemmain1}.   
\end{pf}
Thanks to estimate \eqref{claim1},  we can improve the  estimate \eqref{E011} related to the 
 control of  the time derivative of the  functional $E(w(s),s)$. More precisely, we prove  the following lemma:
\begin{lem}\label{2018lem31} There exists  $\widehat{S}_2>\widehat{S}_1$
 such that   for all  
$s\geq  \widehat{s}_2=\max (-\log T, \widehat{S}_2)$, we have 
\begin{align}\label{E01}
\frac{d}{ds}E(w(s),s)\le &-  \frac{1}{2}\iint (\partial_{s}w)^2\w\y
%\\
%&
+ 
 \frac{M_{7}}{s^{a+\frac74}}\iint |w|^{p+1}\log^a(2+\p^2w^2)\w \y\no\\
&+\frac{C}{s^2}\ibint  w^2\w \y   +\frac{M_{7}}{s^{\frac74}},
 \end{align}
where, $M_7$ depends on $ p, a, N$  and
$\|w(  \widehat{s}_1)\|_{H^1}.$
\end{lem}
\begin{pf}
By using the  additional information obtained in  \eqref{claim1}, we are going to refine the estimate related to $\Sigma_{1}^{2}(s)$ and $\Sigma_{1}^{3}(s)$ defined in \eqref{E00}. Let us mention that the estimate  
\eqref{sigma13} related to $\Sigma_{1}^{1}(s)$ defined in \eqref{E00} is acceptable and does not need any improvement.
 More precisely,  we write 
\begin{align}
\Sigma_{1}^{2}(s)+\Sigma_{1}^{3}(s)=&\frac{p+1}{p-1}
e^{-\frac{(p+1)s}{p-1}}s^{\frac{2a}{p-1}}\iint\big( F(\p w)-\frac{\ps wf(\p w)}{p+1}\big)\w \y\no\\
&-\frac{2a}{p-1}e^{-\frac{(p+1)s}{p-1}}s^{\frac{2a}{p-1}-1}\iint \big( F(\p w)-\frac{\p wf(\p w)}{2}\big)\w \y\no.
 \end{align}
We  attempt to group  the main terms together. 
A straightforward computations implies  that
\begin{equation}\label{2018id2}
\Sigma_{1}^{2}(s)+\Sigma_{1}^{3}(s)=\chi_1(s)+\chi_2(s),
\end{equation}
where
\begin{align}
\chi_{1}(s)=&
\frac{a}{(p+1)s^{a+1}}\iint  {|  w|^{p+1}}\log^{{a-1}}(2+\p^2 w^2  )\Big(\log (2+\p^2w^2 )-\frac{2s}{p-1}\Big)\w \y,
\label{2018id5}\\
\chi_{2}(s)=&\frac{e^{-\frac{(p+1)s}{p-1}}}{p-1}s^{\frac{2a}{p-1}}\iint \Big((p+1) F_2(\p w)-\frac{a}{s} F_1(\p w)-\frac{a}{s}F_2(\p w)\Big)\w \y,\label{2018id8}
\end{align}
where $F_1$ and $F_2$ are defined by \eqref{defF2} and  \eqref{defF123}.
\medskip

Note that, in \eqref{2018id2}
 we  grouped  the main terms together.  In fact, it is easy to  control the terms  $\chi_{2}(s)$. However, the  control of  the  term $\chi_{1}(s)$
  needs the  use of  the  additional information obtained in Lemma \ref{lemmain1}.
%We would like now to find   an estimate for the term  $\chi_{1}(s) $. 
 More precisely,
%
%\medskip
%{\it Claim 1.} 
%There exists  $S_1>0 $ such that   for all   
%  $s \geq -\log (T^*(x)-t_0(x_0))$, we have  
%\begin{align}\label{claim2}
% \chi_{1}(s)\le &
% \frac{K_{14}}{s^{a+\frac74}}\iint |w|^{p+1}\log^a(2+\p^2w^2)\w \y + \frac{K_{14}}{s^{\frac74}}.
% \end{align}
%{\it Proof:} 
for all     
$s\geq  \widehat{s}_1=\max (-\log T,\widehat{S}_1)$, we   divide $\er^N$ into two parts
 \begin{equation}\label{27nov1}
A_{1}(s)=\{y \in B\,\,|\,\, \ps w^2(y,s)\leq  1\}\,\,{\rm and }\,\,A_{2}(s)=\{y \in B
\,\,|\,\, \ps w^2(y,s)\ge  1\}.
\end{equation}
Accordingly, we write  $\chi_1(s)=\chi_1^1(s)+\chi_1^2(s)$, where
\begin{align}
\chi_1^{1}(s)=&
\frac{a}{(p+1)s^{a+1}}\int_{A_1(s)}  {|  w|^{p+1}}\log^{{a-1}}(2+\p^2 w^2  )\Big(\log (2+\p^2w^2 ) -\frac{2s}{p-1}\Big)\w 
,%\label{130}
\nonumber\\
\chi_1^{2}(s)=&
\frac{a}{(p+1)s^{a+1}}\int_{A_2(s)}  {|  w|^{p+1}}\log^{{a-1}}(2+\p^2 w^2  )\Big(\log (2+\p^2w^2 ) -\frac{2s}{p-1}\Big)\w \y.%\label{131}
\nonumber
\end{align}

On the one hand, by using 
 the definition of the set $A_1(s)$ given  in \eqref{27nov1},   we get, for all $s\geq  \widehat{s}_1$, % for all $s\geq  \TT,$
\begin{equation}\label{16dec1}
 | w|^{p+1}\log^{{a}}(2+\p^2 w^2  )\le C \p^{-\frac{p+1}2}(s)\log^{|a|}(2+\ps) \le Ce^{-\frac{s}2}.
\end{equation}
From \eqref{16dec1}  and  the fact that 
%\begin{equation}\label{16dec2}
 $1-\frac{2s}{(p-1)\log (2+\p^2 w^2)}
 \le 1$,
%\end{equation}
 we get
\begin{equation}\label{93}
\chi^1_1(s) \le Ce^{-\frac{s}2}.
\end{equation}

On the other hand, by using  the definition of the $\ps$ given by \eqref{defphi}, we write the identity
\begin{equation}\label{16dec102}
\log (2+\p^2 w^2)
 -\frac{2s}{p-1}=\log (2\p^{-2}+ w^2)-\frac{2a\log s}{p-1}.
\end{equation}
Now,    by using the inequality $\phi  (s)\ge 1 $  and  \eqref{16dec102}, we write  for all 
for all $s\geq  \widehat{s}_1$, 
% $s \geq \max ( -\log (T-t_0), S_0)$, 
\begin{equation}\label{16dec10}
\log (2+\p^2 w^2)
 -\frac{2s}{p-1}\le  \log (2+w^2)+C\log s.
\end{equation}
Also, by using the definition of the set $A_2(s)$ defined in \eqref{27nov1}, we can write 
 for  all $s\geq  \widehat{s}_1$,    if $y\in A_{2}(s)$, we have 
\begin{equation}\label{16dec12}
\log(2+\p^2w^2)\ge \log(\ps)\ge \frac{2s}{p-1}-\frac{a \log s}{p-1}.
\end{equation}
Clearly, the exists   $S_2>S_1 $ such that   for all    
  $s \geq  S_2$, we have   $  \frac{2s}{p-1}-\frac{a \log s}{p-1}\ge  \frac{s}{p-1}$. Therefore,
 by exploiting  \eqref{16dec10} and  \eqref{16dec12} we have  
for all $s\geq  \widehat{s}_2=\max (-\log T, \widehat{S}_2)$, 
% for all  $s \geq \max ( -\log (T^*(x)-t_0(x_0)), S_1)$, 
\begin{align}
\chi_1^{2}(s)\le &
\frac{C}{s^{a+2}}\int_{B}  {|  w|^{p+1}}\log^{{a}}(2+\p^2 w^2  )\log (2+w^2 ) \w \y\nonumber\\
&+\frac{C\log s}{s^{a+2}}\int_{B}  {|  w|^{p+1}}\log^{{a}}(2+\p^2 w^2  )\w \y.\label{131bis}
\end{align}
Note that, by using  the fact $\chi_1(s)=\chi_1^{1}(s)+\chi_1^{2}(s)$,  \eqref{claim1}, \eqref{93}  and  \eqref{131bis}, we get for all $s\geq  \widehat{s}_2=\max (-\log T, \widehat{S}_2)$, 
\begin{equation}\label{13janva1}
\chi_1(s)\le 
\frac{M_{8}}{s^{a+\frac74}}\iint   {|  w|^{p+1}}\log^{{a}}(2+\p^2 w^2 )\w \y+\frac{M_{8}}{s^{\frac74}}.
\end{equation}
%Finally, it remains only to control the term $\chi_2(s)$.
Thanks to     \eqref{equiv2}  and  \eqref{equiv3}, we write
\begin{equation}\label{15dec1}
\frac{1}{s}| F_1(\p w)|+| F_2(\p w)|\le   C+C \frac{ \p w}{s^2}f(\p w).
\end{equation}
By \eqref{E00}, \eqref{15dec1} and \eqref{id1},  we have,  for all $s\geq  \widehat{s}_1$, 
\begin{equation}\label{sigma11dec18}
\chi_2(s)\le 
\frac{C}{s^{a+2}}\iint |w|^{p+1}\log^a(2+\p^2w^2)\w \y+  C e^{-\frac{s}2}.
\end{equation}
The result \eqref{E01} derives immediately from  \eqref{E00}, \eqref{sigma13},  \eqref{13janva1}, \eqref{sigma11dec18},  and  the identity \eqref{2018id2},
which ends the proof of Lemma \ref{2018lem31}
\end{pf}

With Lemmas \ref{LemJ_0} and  \ref{2018lem31},  we are in a position to  prove
Theorem \ref{t1}.

\bigskip

  {\it{Proof of Theorem \ref{t1}:}
By exploiting the defintion of $L_0(w(s),s)$ in \eqref{5jan1},
we  can write easily,  for all $s\geq  \widehat{s}_2=\max (-\log T,\widehat{S}_2)$, 
% $s \geq \TT$,
\begin{equation}\label{14jan1}
\frac{d}{ds}L_0(w(s),s)=\frac{d}{ds}E(w(s),s) + \frac{1}{\sqrt{s}}\frac{d}{ds}J(w(s),s)-  \frac{1}{2s\sqrt{s}}J(w(s),s),
\end{equation}
where $J(w(s),s)=\frac1{s}\iint   w^2\w \y$.
Lemmas \ref{LemJ_0} and  \ref{2018lem31} 
 allows to prove that for all $s\geq  \widehat{s}_2=\max (-\log T,\widehat{S}_2)$, 
% for all $s \geq \TT$, 
we have 
\begin{align*}
\frac{d}{ds}L_0(w(s),s)\le &-  \frac{1}2\iint (\partial_{s}w)^2{\w}\y+ \frac{p+3}{2s\sqrt{s}}L_0(w(s),s)\no\\
%&-\frac{1}{s\sqrt{s}}(\frac{p+7}{4}-\frac{C}{\sqrt{s}})\iint   (\partial_{s}w)^2\w \y\\
%&-\frac1{s\sqrt{s}}(\frac{p-1}{4}-\frac{C}{\sqrt{s}})\iint (\partial_y w)^2(1-y^2)\w\y
%\no\\
&-\frac1{s^{a+\frac32}}\Big(\frac{p-1}{2(p+1)}-  \frac{M_7}{s^{\frac14}}-\frac{C}{s}\Big)\iint |w|^{p+1}\log^a(2+\p^2w^2)  
\w\y\\
&- \frac{1}{s\sqrt{s}}\Big( \frac{p+1}{2(p-1)} -\frac{C}{\sqrt{s}}\Big)\iint  w^2\w \y  +\frac{M_{7}}{s^{\frac74}}+Ce^{-s}.
  \end{align*}
Again,  choosing  $\widehat{S}_3> \widehat{S}_2$ large enough,  this  implies that  for all 
for all $s\geq  \max (-\log T,\widehat{S}_3)$, 
%$s \geq \max (-\log T,S_7)$, 
we have 
\begin{equation}\label{17dec1}
\frac{d}{ds}L_0(w(s),s)\le -\frac1{2}\iint (\partial_{s}w)^2{\w}\y+ \frac{p+3}{2s\sqrt{s}}L_0(w(s),s)+  \frac{M_{9}}{s^{\frac74}}.
\end{equation}
Recalling that,
$$L(w(s),s)=\exp\Big(\frac{p+3}{\sqrt{s}}\Big) L_0(w(s),s)+\frac{\theta}{s^{\frac34}}.$$  
   we get from straightforward computations 
  \begin{equation}\label{17dec2}
  \frac{d}{ds}L(w(s),s) =-\frac{p+3}{2s\sqrt{s}}\exp\Big(\frac{p+3}{\sqrt{s}}\Big) L_0(w(s),s)+\exp\Big(\frac{p+3}{\sqrt{s}}\Big)\frac{d}{ds}L_0(w(s),s)-\frac{4\theta}{3s^{\frac74}}.
\end{equation}
Therefore, estimates $\eqref{17dec1}$ and $\eqref{17dec2}$ lead to the following crucial estimate:
 \begin{equation}
 \frac{d}{ds}L(w(s),s)\leq -\frac12\exp\Big(\frac{p+3}{\sqrt{s}}\Big) \iint (\partial_{s}w)^2 {\w}\y
+\Big(M_9\exp\Big(\frac{p+3}{\sqrt{s}}\Big)-\frac{4\theta}3 \Big) \frac{1}{s^{\frac74}}.
\end{equation}
Since we have $1\leq \exp\Big(\frac{p+3}{\sqrt{s}}\Big)\leq \exp\Big(\frac{p+3}{\sqrt{
\widehat{S}_3}}\Big)$, 
we then choose $\theta$ large enough, so that $
M_9\exp\Big(\frac{p+3}{\sqrt{s}}\Big)-\frac{4\theta}3
 \leq 0$, which yields,  for all 
 $s\geq  s_3=\max (-\log T,\widehat{S}_3)$, 
$$\frac{d}{ds}L(w(s),s)\leq -\frac12 \iint (\partial_{s}w)^2 {\w}\y.$$
A simple integration between $s$ and $s+1$ ensures the result.
This concludes the proof
of  Theorem \ref{t1}.}
\Box

\medskip

% \noindent
 We now claim the following lemma:
\begin{lem}\label{L19}
 There exist $M_{10}>0$ 
and
 $\hat S_{4}\geq \hat S_3$ such that, 
we have  for all $s\ge \max( \hat{S}_4,-\log T)$
 \begin{equation}\label{posi1}
N_{m_0}(w(s),s)\geq -M_{10}.
\end{equation}

\end{lem}
\begin{pf}
The argument is the same  as  the similar part in  Proposition \ref{proplyap}. 
\end{pf}
\subsection{Proof of Theorem \ref{t2} }\label{3.3}
 As in   \cite{GMSiumj04}, by combining Theorem  \ref{t1}
 and  Lemma \ref{L19}  we get the following bounds:

\begin{coro}\label{4dec1} 
  For all    $s\geq  \max (-\log T,\widehat{S}_4)$, 
 we have
\begin{equation}\label{3dec0}
-M_{11}\leq L(w(s),s) \leq  M_{11},
\end{equation}
\begin{equation}\label{3dec1}
 \int_{s}^{s+1}\iint\Big(|\grad w|^2+(\partial_sw)^{2}+w^2\Big) \w  {\mathrm{d}}y{\mathrm{d}}\tau\leq M_{12},
\end{equation}
\begin{equation}\label{3dec2}
\frac1{s^a}\int_{s}^{s+1}\int_{\er^N}   {|  w|^{p+1}}\log^{{a}}(2+\p^2 w^2  )\w \y
{\mathrm{d}}\tau\leq M_{13}.
\end{equation}
%Moreover, there exists a time
%$\widehat{S}_5\geq  \widehat{S}_4$,  
%such that for all  $s\geq  \max (-\log T,\widehat{S}_5)$, 
\begin{equation}\label{3dec21}
\iint w^2\w   \y\leq M_{14},
%\|w(s)\|_{L_\rho^2(\mathbb{R}^n)}^2 \leq J_1,
\end{equation}
\begin{equation}\label{3dec22}
\frac1{s^a}\int_{\er^N}   {|  w|^{p+1}}\log^{{a}}(2+\p^2 w^2  )\w \y
\leq C\iint |\grad w|^2 \w  \y +M_{15},
\end{equation}
\begin{equation}\label{3dec23}
 \iint |\grad w|^2 \w  {\mathrm{d}}y
\leq C\sqrt{\iint (\partial_sw)^2 \w  {\mathrm{d}}y} +M_{16},
\end{equation}
\begin{equation}\label{jan181bisbis}
 \iint |\grad w|^2 \w  \y \le 
\frac{C}{s^a}\int_{\er^N}   {|  w|^{p+1}}\log^{{a}}(2+\p^2 w^2  )\w \y
 +M_{17},
\end{equation}
\begin{equation}\label{3dec24}
\int_{s}^{s+1}\Big( \iint |\grad w|^2 \w  {\mathrm{d}}y\Big)^2
\leq M_{18},
\end{equation}
\begin{equation}\label{3dec25}
\frac1{s^{2a}}\int_{s}^{s+1}\Big(\int_{\er^N}   {|  w|^{p+1}}\log^{{a}}(2+\p^2 w^2  )\w \y\Big)^2
{\mathrm{d}}\tau\leq M_{19},
\end{equation}
where $M_{11}, M_{12},M_{13},... M_{19} $ depend on $ p, a, N,  s_3 =  \max (-\log T,\widehat{S}_3)$  and
$\|w( s_3)\|_{H^1}$.
\end{coro}

\bigskip

Let us denote that, the estimates obtained  in the above corollary  are similar to the   Corollary \eqref{19dec3}  except for the  presence of the  term $K_is^{b+1}$ instead of $M_i$. Consequently,
 following the proof  of  Proposition \ref{19dec3bis1}  line by line
 we are in position  to prove the following:

\begin{prop}\label{19dec3bis} 
For all  $q \geq 2$,   $\varepsilon >0$ and $R>0$ there exist   $\varepsilon_6=\varepsilon_6(q,R)>0$,
there  exists a time
$\widehat{S}_5(q,R,\varepsilon)\geq  \widehat{S}_4$,  
such that for all  $s\geq  \max (-\log T,\widehat{S}_5)$, 
 we have
\begin{equation*}
(E_{q,R,\vv})\ \ \int_s^{s+1} \|w(\tau )\|_{L^{\pp +1}(\mathbf{B}_R)}^{(\pp+1)q} {\mathrm{d}\tau}  \leq  M_{20}(q,R,\varepsilon ),
\end{equation*}
where  $M_{20}(q,R,\varepsilon )$ depends on $ p, a, N,q, R,\varepsilon, 
  s_3 =  \max (-\log T,\widehat{S}_3)$  and
$\|w( s_3)\|_{H^1}$.
\end{prop}

Finally,  we are in position to prove Theorem \ref{t2} by   exploiting Lemma  \ref{lemm:intpola}  and  Lemma \ref{prop:regpar}.

{{\it {Proof  of Theorem \ref{t2} }}:
First,
%For this polynomial estimate, 
we  use  \eqref{3dec1}, Proposition \ref{19dec3bis} and apply Lemma  \ref{lemm:intpola} with $\alpha = q(p- \frac{ \varepsilon}2+1)$, $\beta =p- \frac{ \varepsilon}2+1$, $\gamma = \delta  = 2$ to get that, for all $s \geq \max  (-\log T,\widehat{S}_5)$,
\begin{equation}\label{equ:tmp2key11}
%\sup_{\tau \in [s,s+1]}
 \|w(s)\|_{L^\lambda(\mathbf{B}_R)} \leq M_{21}(q,R,\vv) , \ \   \forall \lambda <    p- \frac{ \varepsilon}2+1 - \frac{
p- \frac{ \varepsilon}2
-1}{q + 1},\ \   \forall  \vv  \in (0,p-1),\ \  \forall q\ge2.
\end{equation}
Hence, for all $\varepsilon \in (0,p-1)$, we have  $q=\frac{4p-4-\varepsilon}{\varepsilon}\ge 2$. Therefore, the estimate \eqref{equ:tmp2key11}  implies
\begin{equation}\label{k2}
\sup_{\tau \in [s,s+1]} \|w(\tau)\|_{L^{p+1-\varepsilon}(\mathbf{B}_R)} \leq M_{22}(R,\vv)., \qquad \forall \varepsilon \in (0,p-1).
\end{equation}
 Let us recall the equation in $w$: 
\begin{align}\label{Abis}
\partial_{s}w&=  \Delta w - \frac{1}{2}y. \nabla w -\frac{1}{p-1}(1-\frac{a}{s})w+
e^{-\frac{ps}{p-1}}s^{\frac{a}{p-1}} f(\ps w),
\end{align}
where   $\ps$ and $f$ are given in \eqref{defphi} and \eqref{deff}.

\medskip

We now apply Lemma \ref{prop:regpar} to $w,$ with $b =b(y)= \frac{1}{2}y$ and 
$$H(y,s,w) = -\frac{1}{p-1}(1-\frac{a}{s})w+
e^{-\frac{ps}{p-1}}s^{\frac{a}{p-1}} f(\ps w).$$
From \eqref{equiv44}, we see that, for all $\varepsilon \in (0,p-1)$, we have 
$$|H(y,s,w)| \leq C(\varepsilon )(|w|^{p-1+\varepsilon } + 1)(|w| + 1) , \quad \forall s \geq \max  (-\log T ,\widehat{S}_5).$$
Let $\lambda_1= p+1-\varepsilon  $, $\alpha_1=\frac{\lambda_1}{p-1+\varepsilon}$  and $\beta_1=\frac1{\varepsilon}$.
Thus, the first identity in \eqref{equ:conLemint} holds with $g(y,s,w) = C(\varepsilon )(|w(y,s)|^{p-1+\varepsilon } + 1)$.
Since $p  < \frac{N + 2}{N - 2}$, then    we can choose  $\varepsilon_7\le \varepsilon_6$ small enough, such that   the conditions $\frac{1}{\beta_1} + \frac{N}{2\alpha_1} 
< 1$  and $\alpha_1\ge 1$ hold. Moreover,  for all $s \geq \max  (-\log T ,\widehat{S}_5)$ we have
\begin{align}\label{k1}
\int_s^{s+1} \|g(\tau)\|_{L^{\alpha_1}(\mathbf{B}_R)}^{\beta_1}\t\le C+C
\int_s^{s+1} \left(\int_{\mathbf{B}_R} |w(y,\tau)|^{\lambda_1}dy\right)^{\frac{1}{\varepsilon_7 \alpha_1}}\t\nonumber  \\
\le C+C\Big( \sup_{\tau \in [s,s+1]} \|w(\tau)\|_{L^{\lambda_1}(\mathbf{B}_R)} \Big)^{\frac{p-1+\varepsilon_7}{\varepsilon_7}}.
%\le C+C s^{(b+1)q\beta_1(p-1+\varepsilon)}
\end{align}
By exploiting \eqref{k1} and \eqref{k2}, we deduce that
\begin{align}\label{k3}
\int_s^{s+1} \|g(\tau)\|_{L^{\alpha_1}(\mathbf{B}_R)}^{\beta_1}\t
\le  M_{23}(R,\varepsilon_7).
%\le C+C s^{(b+1)q\beta_1(p-1+\varepsilon)}
\end{align}
%\quad \text{for some $C_1 > 0$,}$$
%for some $\alpha_1$ and $\beta_1$ satisfying $\frac{1}{\beta_1} + \frac{n}{2\alpha_1} < 1$.\\
Then the second condition 
in  \eqref{equ:conLemint} holds. Therefore,
\begin{equation}\label{equ:tmp2key1}
\|w(s)\|_{L^{\infty}(\mathbf{B}_{\frac{R}4})} \leq  M_{24}(R), \quad  \forall s \geq \max  (\tau_0-\log T, \tau_0+\widehat{S}_5)),
\end{equation}
 for some $\tau_0 \in (0,1)$.
By \eqref{equ:tmp2key1}, we write
\begin{equation}\label{equ:tmp2key22}
|w_{x_0}(0,s)|  \leq  M_{25}, \quad  \forall s \geq \max  (\tau_0-\log T, \tau_0+\widehat{S}_5),
\end{equation}
 for some $\tau_0 \in (0,1).$
 From the fact that the above estimate is independent of $x_0$ and the definition of  $w_{x_0}$ given by \eqref{scaling}, we infer
\begin{equation}\label{equ:tmp2key221}
|w(y,s)|  \leq  M_{25}, \quad  \forall y\in \er^N \quad \forall s \geq \max  (1-\log T, 1+\widehat{S}_5).
\end{equation}
This concludes the proof of Theorem \ref{t2}. \Box

%%%%%%%%%%%%%%
%%%%%%%%%%%%%%
%%%%%%%%%%%%%%
%%%%%%%%%%%%%%
%%%%%%%%%%%%%%
%%%%%%%%%%%%%%
%%%%%%%%%%%%%%
%%%%%%%%%%%%%%
%%%%%%%%%%%%%%
%%%%%%%%%%%%%%
%%%%%%%%%%%%%%
%%%%%%%%%%%%%%
%%%%%%%%%%%%%%
%%%%%%%%%%%%%%
%%%%%%%%%%%%%%
%%%%%%%%%%%%%%
%%%%%%%%%%%%%%
%%%%%%%%%%%%%%

%%%%%%%%%%%%%%%%%%%%%%%%%%%%%%%%%%%%%%%%%%%%%%%%%%%%%%%%%%%%%%%%%%%%%%%%%%
%%%%%%%%%%%%%%%%%%%%%%%%%%%%%%%%%%%%%%%%%%%%%%%%%%%%%%%%%%%%%%%%%%%%%%%%%%
%%%%%%%%%%%%%%%%%%%%%%%%%%%%%%%%%%%%%%%%%%%%%%%%%%%%%%%%%%%%%%%%%%%%%%%%%%
%%%%%%%%%%%%%%%%%%%%%%%%%%%%%%%%%%%%%%%%%%%%%%%%%%%%%%%%%%%%%%%%%%%%%%%%%%

%%%%%%%%%%%%%%%%%%%%%%%%%%%%%%%%%%%%%%%%%%%%%%%%%%%%%%%%%%%%%%%%%%%%%%%%%%
%%%%%%%%%%%%%%%%%%%%%%%%%%%%%%%%%%%%%%%%%%%%%%%%%%%%%%%%%%%%%%%%%%%%%%%%%%

 \appendix

 \section{Appendix A.}
We recall the interpolation  result from Cazenave and Lions \cite{CLcpde84} and the interior regularity theorem in \cite{GKcpam85}.

\begin{lem}[\textbf{Interpolation technique, Cazenave and Lions \cite{CLcpde84}}] \label{lemm:intpola} Let $t_0>0$.  Assume that 
$$v \in L^\alpha\left( [t_0,t_0+1]; L^\beta(\mathbf{B}_R) \right), \; \partial_tv \in L^\gamma\left( [t_0,t_0+1]; L^\delta(\mathbf{B}_R) \right)$$
for some $1 < \alpha, \beta, \gamma, \delta < \infty$. Then 
$$v \in \mathcal{C}\left([t_0,t_0+1]; L^\lambda(\mathbf{B}_R) \right)$$
for all $\lambda  < \lambda_0 = \frac{(\alpha + \gamma')\beta\delta}{\gamma'\beta + \alpha\delta}$ with  $\gamma' = \frac{\gamma}{\gamma - 1}$, and satisfies
$$\sup_{t \in [t_0,t_0+1]} \|v(t) \| _{L^\lambda(\mathbf{B}_R)} \leq C \int_{t_0}^{t_0+1} \left(\|v(\tau)\|_{L^\beta(\mathbf{B}_R)}^\alpha + \|\partial_\tau v(\tau)\|_{L^\delta(\mathbf{B}_R)}^\gamma \right)\t$$
for $\lambda < \lambda_0$. The positive constant $C$ depends only on $\alpha, \beta, \gamma, \delta , N$ and $R$.
\end{lem}

\noindent The second one is an interior regularity result for a nonlinear parabolic equation:

\begin{lem}[\textbf{Interior regularity}] \label{prop:regpar}{\  \\  }
 Let $v(x,t)\in L^\infty\big((0,+\infty), L^2(\mathbf{B}_R)\big) \cap L^2\big((0,+\infty),H^1(\mathbf{B}_R)\big)$ which satisfies  
\begin{equation}\label{equ:vinger}
v_t - \Delta v + b. \nabla v = H,\quad (x,t) \in Q_R = \mathbf{B}_R \times (0, +\infty),
\end{equation}
where $R > 0$, $|b(x,t)| \leq \mu_1$ in $Q_R$  and $|H(x,t,v)| \leq g(x,t)(|v| + 1)$ with
\begin{equation}\label{equ:conLemint}
\int_{t}^{t +1} \left\|g(\tau)\right\|^{\beta'}_{L^{\alpha'}(\mathbf{B}_R)}d\tau  \leq \mu_2, \quad \forall t \in(0, +\infty),
\end{equation}
and $\frac{1}{\beta'} + \frac{N}{2\alpha'} < 1$, % $\beta'>\frac{2}{1-\eta}$,
 and $\alpha' \geq 1$. If 
\begin{equation}\label{eq:conLeminreg1}
%\int_{t}^{t +1}
\int_{t}^{t +1}  \|v(\tau)\|^2_{L^2(\mathbf{B}_R)}d\tau  \leq \mu_3,\quad \forall t \in(0, +\infty),
\end{equation}
and $\mu_1$, $\mu_2$ and $\mu_3$ are uniformly bounded in $t$, then there exists a positive constant $C$ depending only on $\mu_1$,  $\mu_2$, $\mu_3$, $\alpha'$, $\beta'$, $N$, $R$ and $\tau \in (0,1)$ such that
$$|v(x,t)| \leq C,\quad \forall(x,t) \in \mathbf{B}_{R/4} \times (\tau, +\infty).$$
\end{lem}

\bigskip

 \section{Some elementary lemmas.}
 Let $f$, $F$, $F_2$  be the functions defined in  \eqref{deff}, \eqref{defF}  and  \eqref{defF123}.  
Clearly, we have 
\begin{lem}\label{FFF} \ {\  }Let $q>1$,\\ 
\begin{align}
\int_0^u|v|^{q-1}v\log^{{a}}(2+v^2 )\v  \sim&  \frac{| u|^{q+1}}{q+1}\log^{{a}}(2+u^2  ),\quad  \text{ as } \;\; |u| \to \infty,
\label{estF0}\\
F(u)  \sim &\frac{uf(u)}{p+1} \quad  \text{ as } \;\; |u| \to \infty,\label{estF}\\
%F_1(u)\sim& -\frac{2a}{(p+1)^2}\frac{uf(u)}{\log(2+u^2)}\quad  \text{ as } \;\; |u| \to \infty, \label{estF2}\\
F_2(u)\sim &\frac{Cuf(u)}{\log^2(2+u^2)}\quad  \text{ as } \;\; |u| \to \infty.\label{estF3}
%F_4(\beta w)\le Ce^{-\frac4{p-1}} (1+F(\beta w))\label{estf3}
\end{align}
\end{lem}
\begin{proof} 
 See Lemma A.1 %page 87
 in \cite{HZjmaa2020}.
\end{proof} 
\Box

Thanks to \eqref{estF0}, \eqref{estF} and \eqref{estF3}, we will give the  first and the second order terms   in the   expansion of    the nonlinearity $F(x)$ defined in \eqref{defF},  when $|x|$ is  large enough.
More precisely,  we 
%\begin{equation}\label{defF123}
%F(x)=\frac{xf(x)}{p+1}+F_1(x)+ F_2(x),
%\end{equation}
%where 
%\begin{equation}
%F_1(x)= -\frac{ 2a} {(p+1)^2}{| x|^{p+1}}\log^{{a-1}}(2+x^2  ).\label{defF2}
%\end{equation}
  now state the following estimates: 
%on the nonlinear  term $F$  defined in \eqref{defF}, the first approximation $F_1$ defined in
 %\eqref{defF2} and $F_2$ defined in  \eqref{defF123},  which will be needed later:
\begin{lem}\label{lemm:esth}% Let $F,  F_1,  F_2$  be the function defined in \eqref{defF}. 
For all $s \geq 1$,  for all $z\in \er$,
%$i)$ 
\begin{align}
 C^{-1} \ps z f(\ps z))\le C+F\left(\ps z)\le C (1+\ps z f(\ps z)\right),\label{equiv1}\\
F_1(\ps z)\le C+C\frac{ \ps z}{s}f(\ps z),\quad\quad\label{equiv2}\\
F_2(\ps z)\le C+C \frac{ \ps z}{s^2}f(\ps z),\quad\quad\label{equiv3}\\
e^{-\frac{ps}{p-1}}s^{\frac{a}{p-1}}   |f(\ps z)|\le C (\varepsilon) + C   |z|^{p+\varepsilon  }, \quad\quad \forall  \varepsilon \in (0,p-1),
\label{equiv44}\\
   |z|^{\pp }\le C e^{-\frac{ps}{p-1}}s^{\frac{a}{p-1}} | f(\ps z)|+C (\varepsilon), \quad\quad \forall  \varepsilon \in (0,p-1),
\label{equiv44bis}\\
e^{-\frac{(p+1)s}{p-1}}s^{\frac{2a}{p-1}}  F(\ps z)\le C (\varepsilon) + C  |z|^{p+ \varepsilon+1}, \quad\quad \forall  \varepsilon \in (0,p-1),  
\label{equiv4}\\
  |z|^{\pp +1}\le
e^{-\frac{(p+1)s}{p-1}}s^{\frac{2a}{p-1}}  F(\ps z) +C (\varepsilon),\quad \quad \forall   \varepsilon \in (0,p-1),
\label{equiv4bis}
\end{align}
where
$\phi$,  $F$, $F_1$ and $F_2$  are given in  \eqref{defphi},    \eqref{defF}, 
 \eqref{defF2} and   \eqref{defF123}.
% which will be needed later:

% $\bar p =p+1 $ if $N=1,2$ and $\bar p=p+\frac{2}{N-2}-\frac2{N-1}$, if $N\ge 3$.
%\begin{equation}\label{defpb}
% \bar p =
%%\begin{cases}
%%p+1 & \text{ if } N=1,2, \\
%%p+\frac{2}{N-2}-\frac2{N-1} & \text{ if } N \ge 3.
%%\end{cases}
%\end{equation}
%$ii)$ 
%\begin{equation*}
%\left|(p+1)e^{-\frac{(p+1)s}{p-1}} H\left(e^\frac{s}{p-1}z\right) -  e^{-\frac{ps}{p-1}} h\left(e^\frac{s}{p-1}z\right)z\right| \leq \frac{C_0}{s^{a+1}}(|z|^{p+1} + 1).
%\end{equation*}
\end{lem}
\begin{proof} Note that \eqref{equiv1} obviously follows from \eqref{estF}. In order to derive estimates \eqref{equiv2} and \eqref{equiv3}, considering the first case $z^2\ps \geq 4$, then the case $z^2\ps \leq 4$, we would obtain
 \eqref{equiv2} and \eqref{equiv3} by using \eqref{estF0}, \eqref{estF}  and\eqref{estF3}.
Similarly, by taking into account  the  inequality  $\log^a (2+u^2)\le C(\varepsilon )+  C(\varepsilon) |u|^{\varepsilon }$ ,  we  conclude easily 
\eqref{equiv44}, \eqref{equiv44bis}, \eqref{equiv4} and \eqref{equiv4bis}.
%$ii)$  Replacing $\xi$ by $\beta(s)z$ and using \eqref{equ:estBah}, we then derive $ii)$.
 This ends the proof of Lemma \ref{lemm:esth}.\Box
\end{proof}

\section{Proof of Proposition \ref{prop:upElc}}\label{ap:upELc}
 Let us first derive the upper bound for $\mathcal{E}_\psi$.\\
\begin{proof}[\textbf{Proof of the upper bound for $\mathcal{E}_\psi$}]
 Multiplying $\eqref{A}$ by $ \partial_{s} w \psi^2\w$ and integrating over  $\er^N,$ we obtain
\begin{align}
\frac{d}{ds}\mathcal{E}_\psi(w(s),s)=& -  \ibint (\partial_{s}w)^2\psi^2\w\y
-  2\ibint \partial_{s}w\nabla w. \nabla \psi \psi \w\y\nonumber\\
&+\underbrace{\frac{a}{(p-1)s}\ibint  w\partial_{s}w\psi^2\w\y}_{\Sigma^1_{2}(s)}\nonumber\\
&+\underbrace{\frac{p+1}{p-1}
e^{-\frac{(p+1)s}{p-1}}s^{\frac{2a}{p-1}}\ibint\big( F(\p w)-\frac{\p wf(\p w)}{p+1}\big)\psi^2\w \y}_{\Sigma^2_{2}(s)}\no\\
&\underbrace{-\frac{2a}{p-1}e^{-\frac{(p+1)s}{p-1}}s^{\frac{2a}{p-1}-1}\ibint \big( F(\p w)-\frac{\p wf(\p w)}{2}\big)\psi^2
\w \y}_{\Sigma^3_{2}(s)}.
 \end{align}
Proceeding similarly as for the terms $\Sigma_{1}^{1}(s)$, $\Sigma_{1}^{2}(s)$ and  $ \Sigma_{1}^{3}(s)$
 defined in \eqref{E00}, we get 
\begin{align}\label{E011bis2020}
\frac{d}{ds}\mathcal{E}_\psi(w(s),s)\le& -  \frac{1}{2}\ibint\psi^2 (\partial_{s}w)^2\w\y-  2\ibint \partial_{s}w\psi \nabla \psi. \nabla w\w\y\\
&+  \frac{C}{s^{a+1}}\ibint \psi^2|w|^{p+1}\log^a(2+\p^2w^2)\w \y
+\frac{C}{s^2}\iint  \psi^2 w^2\w \y  +C e^{-s}.\nonumber
 \end{align}
Using the fact that $2ab \leq \frac{a^2}{4} + 4b^2$, we obtain 
$$-2 \partial_{s}w\psi \nabla \psi. \nabla w \leq \frac{1}{4}\psi^2 (\partial_sw)^2 + 4 |\nabla \psi|^2 |\nabla w|^2,$$
which implies,  for all $s\ge \T$,
\begin{align}\label{E0112bis}
\frac{d}{ds}\mathcal{E}_\psi(w(s),s)\le&\  \   C \ibint |\grad w|^2\w\y+  \frac{C}{s^{a+1}}\ibint |w|^{p+1}\log^a(2+\p^2w^2)\w \y\nonumber\\
&+\frac{C}{s^2}\iint   w^2\w \y  +C e^{-s},
 \end{align}
where $C = C(a, p, N,  \|\psi\|_{L^\infty}, \|\grad \psi\|_{L^\infty})$.\\

By combining (\ref{E0112bis}), \eqref{18jan1} and \eqref{18jan10000}, we infer  for all $s\ge \TTT  $ 
\begin{align}\label{E0112bis1}
\int_s^{s+1} \frac{d}{ds}\mathcal{E}_\psi(w(\tau),\tau)\t \le  Q_1 s^{b+1}.
 \end{align}
From the definition of $\mathcal{E}_\psi$ given in \eqref{Eloc}, using the fact that,  $F(\phi w)\ge0,$  we have 
\begin{align*}
\mathcal{E}_\psi(w(s),s) & \leq \|\psi\|^2_{L^\infty} \int_{\mathbb{R}^N}\left( \frac{1}{2}|\nabla w|^2 + \frac{1}{2(p-1)}|w|^2 \right)\w \y.
\end{align*}
By   the definition of $H_{m}(w(s),s)$ 
given in \eqref{F0},  exploiting \eqref{18jan1bis}, we write  for all $s\ge \TTT  $ 
\begin{align}
\mathcal{E}_\psi(w(s),s)
& \le C  \left\{H_{m_0}(w(s),s)+\frac{m_0}{2s} \int_{\mathbb{R}^N}w^2\w\y +e^{-\frac{(p+1)s}{p-1}}s^{\frac{2a}{p-1}}  \int_{\mathbb{R}^N}  F(\p w)\w \y \right\}\nonumber\\
& \leq Q_2 s^{b+1} +Ce^{-\frac{(p+1)s}{p-1}}s^{\frac{2a}{p-1}}  \int_{\mathbb{R}^N}  F(\p w)\w \y.\label{31jan1} 
\end{align}
Integrating the inequality \eqref{31jan1} from $s$ to
 $s+1$ and using \eqref{id1}, \eqref{equiv1} and \eqref{18jan10000}  we get,  for all $s\ge \TTT  $ 
\begin{align*}
\int_s^{s+1} \mathcal{E}_\psi(w(\tau),\tau ) \t
 \leq Q_3s^{b+1}.
\end{align*}
 By using the mean value theorem, we derive the existence of $\sigma(s)\in [s,s+1]$ such that
\begin{equation}\label{s1}
 \mathcal{E}_\psi(w(\sigma (s)),\sigma (s))
=\int_{s}^{s+1}
 \mathcal{E}_\psi(w(\tau),\tau){\mathrm{d}}\tau.
\end{equation}
 Let us   write  the identity,  for all $s\ge \TTT  $ 
\begin{align}\label{2018d1}
 \mathcal{E}_\psi(w(s),s)=& \mathcal{E}_\psi(w(\sigma (s)),\sigma(s))  +
\int_{\sigma (s)}^{s}\frac{d}{d \tau} \mathcal{E}_\psi(w(\tau),\tau ){\mathrm{d}}\tau.
\end{align}
By combining (\ref{s1}), \eqref{2018d1} and \eqref{E0112bis1}, we infer,  for all $s\ge \TTT  $ 
\begin{align}\label{AAA}
\mathcal{E}_\psi(w(s),s)=& \leq Q_4s^{b+1}.
\end{align}
This concludes the proof of the upper bound for $\mathcal{E}_\psi$.\\

\noindent It remains to prove the lower bound.

[\textbf{Proof of the lower bound for $\mathcal{E}_\psi$}]\\
 Consider now,   for all $s \geq \T$, $$\mathcal{I}_\psi(w(s),s)=\frac1{s^{b+1}}\ibint w^2 \psi^2 \w\y.$$
 Multiplying equation \eqref{A} with $\psi^2 w,$ integrating on $\mathbb{R}^N$  and using the same argument as  in the proof of  Lemma \ref{LemJ_0} yield 
%Thanks to \eqref{6nov2018}, for  all $s \geq \T$, we have 
\begin{align}\label{6nov20181}
\frac{d}{ds}\mathcal{I}_\psi(w(s),s)\  \ge&\  -\frac{p+3}{s^{b+1}}\mathcal{E}_{\psi}(w(s),s)
+ \frac{1}{2s^{b+1}}(1-\frac{C_4}{s}) \ibint  w^2 \psi^2\w\y\no\\
&\ 
%-\frac{m(p+1)}{2(p-1)s} \ibint  w^2\w\y
+\frac{p-1}{(p+1)s^{a+b+1}}(1-\frac{C_4}{s})\ibint |w|^{p+1}\log^a(2+\p^2 w^2)   \psi^2
\w\y\no\\
& - \frac4{s^{b+1}} \int_{\mathbb{R}^N} w \nabla w.\nabla \psi  \psi  \w \y. 
\end{align}
Therefore, there  exists   $\tilde S_2>S_2$   large enough,  such that  for  all $s \geq \max (-\log T, \tilde S_2)$, we have
\begin{align}\label{1fev5}
\frac{d}{ds}\mathcal{I}_\psi(w(s),s)\  \ge&\ 
%-\frac{m(p+1)}{2(p-1)s} \ibint  w^2\w\y
\frac{p-1}{2(p+1)s^{a+b+1}}\ibint |w|^{p+1}\log^a(2+\p^2 w^2)  
 \psi^2\w\y\no\\
&\   -\frac{p+3}{s^{b+1}}\mathcal{E}_{\psi}(w(s),s)- \frac4{s^{b+1}} \int_{\mathbb{R}^N} w \nabla w.\nabla \psi  \psi  \w \y. 
%- \frac4{ s^{b+1}}\int_{\mathbb{R}^N} \psi w \nabla \psi . \nabla w \w \y.
\end{align}
Furthermore, %tWe now control the new term $\ds{-4 \int_{\mathbb{R}^N} w \nabla w.\nabla \psi  \psi  \w \y}$,
 after some integration by parts, we write
%-4\int_{\mathbb{R}^N} \psi w \nabla \psi . \nabla w \w  \y$ 
\begin{align}\label{1nov3}
&-4 \int_{\mathbb{R}^N} w \nabla w.\nabla \psi  \psi  \w \y
%-4\int_{\mathbb{R}^N} \psi w \nabla \psi . \nabla w \w \y  
=2\int_{\mathbb{R}^N} w^2\div(\psi\w  \nabla \psi  ) \y \nonumber\\
 &=2\int_{\mathbb{R}^N} w^2  |\nabla \psi|^2  \w \y + 2 \int_{\mathbb{R}^N} w^2 \psi \Delta \psi \w \y - \int_{\mathbb{R}^N} w^2\psi y.\nabla  \psi \w \y. 
\end{align}
Thanks to the estimates 
$ \|\psi\|^2_{L^\infty} +\|\Delta \psi\|^2_{L^\infty}  + \|\nabla \psi\|^2_{L^\infty}+ \|y.\nabla \psi\|^2_{L^\infty}\le C$, \eqref{1nov3} and  
\eqref{18jjjjj}, we have  for  all $s \geq \max (-\log T, \tilde S_2)$, 
\begin{align}\label{1fev6}
\Big|-4 \int_{\mathbb{R}^N} w \nabla w.\nabla \psi  \psi  \w \y\Big|
%-4\int_{\mathbb{R}^N} \psi w \nabla \psi . \nabla w \w \y  
 \leq C \int_{\mathbb{R}^N} w^2\w \y
\leq Q_5s^{b+1}.
\end{align}
Using \eqref{1fev5} and \eqref{1fev6},  we obtain  for  all $s \geq \max (-\log T, \tilde S_2),$ 
\begin{align}\label{0fev2}
\frac{d}{ds}\mathcal{I}_\psi(w(s),s)\  \ge&\ 
%-\frac{m(p+1)}{2(p-1)s} \ibint  w^2\w\y
\frac{p-1}{2(p+1)s^{a+b+1}}\ibint |w|^{p+1}\log^a(2+\p^2 w^2)  
 \psi^2\w\y\no\\
&\  -\frac{p+3}{s^{b+1}}\mathcal{E}_{\psi}(w(s),s)-Q_5.
%& - 4 s^{-b}\int_{\mathbb{R}^N} \psi w \nabla \psi . \nabla w \w \y. 
\end{align}

Let us define  the following functional:
\begin{equation}
\mathcal{G}_{\psi}(w(s),s) = \frac{ p+3}{s^{b+1}}\mathcal{E}_{\psi}(w(s),s)+ Q_5,
\end{equation}
where  $\mathcal{G}_{\psi}(w(s),s) $ is defined in \eqref{Eloc}.

\medskip

We claim that the function of $\mathcal{G}_{\psi}(w(s),s) $ is bounded from below by some constant $M$,
where $M$ is a sufficiently large constant that will be determined later.  Arguing by contradiction, we suppose that there exists a time $s^* \geq \max (-\log T, \tilde S_2)$ such that $\mathcal{G}_{\psi}(w(s^*),s^*) \leq - Q$, for some $Q>0$. Then,   we write 
\begin{align}\label{0fev1}
\mathcal{G}_{\psi}(w(s),s)  \leq -Q + \int_{s^*}^s\frac{d}{d\tau} \mathcal{G}_{\psi}(w(\tau),\tau) \t, \qquad   \forall  s \geq s^*.
%& \leq -M + m + C_5(s - s^*).
\end{align}
If we now compute the time derivative of  $\mathcal{G}_{\psi}(w(s),s) $ we get  for all $s \geq s^*,$ 
\begin{align}\label{1fev7}
\frac{d}{ds} \mathcal{G}_{\psi}(w(s),s) =&
\frac{ p+3}{s^{b+1}}\ \frac{d}{ds}\mathcal{E}_{\psi}(w(s),s)-\frac{(b+1) (p+3)}{s^{b+2}}\mathcal{E}_{\psi}(w(s),s).
\end{align}
From the definition of $\mathcal{E}_\psi$ given in \eqref{Eloc}, using \eqref{equiv1}  and \eqref{id1}  we have for all $s \geq s^*,$ 
\begin{equation}
-\frac{(b+1) (p+3)}{s^{b+2}}\mathcal{E}_{\psi}(w(s),s)
%-(b+1) (p+3)s^{-b-2}\mathcal{E}_{\psi}(w(s),s)
\le \frac{C}{s^{a+b+2}}\ibint |w|^{p+1}\log^a(2+\p^2 w^2)   \psi^2
\w\y+Ce^{-s}.\label{2fev1} 
\end{equation}
Thanks to  \eqref{E0112bis1}   we conclude for all $s \geq s^*,$ 
\begin{align}\label{2fev10}
\int_{s^*}^{s} \frac1{\tau^{b+1}}\frac{d}{ds}\mathcal{E}_\psi(w(\tau),\tau)\t \le  Q_6 (s-s^*).
 \end{align}
Moreover, 
from \eqref{18jan10000}, we obtain for all $s \geq s^*,$ 
\begin{align}\label{2fev11}
\int_{s^*}^{s}\frac1{ \tau^{a+b+2}} \ibint |w|^{p+1}\log^a(2+\p^2 w^2)   \psi^2
\w\y\t \le  Q_7(s-s^*).
 \end{align}
Integrating the identity \eqref{1fev7}
 over $[s^*,s]$
and  combining  \eqref{2fev1}, \eqref{2fev10}  and \eqref{2fev11}  we deduce  that
%  \eqref{id1} and  \eqref{2fev1}, we  have
\begin{align}\label{jan2021}
\ \int_{s^*}^s\frac{d}{d\tau} \mathcal{G}_{\psi}(w(\tau),\tau) \t
\leq   Q_8(s - s^*), \qquad   \forall  s \geq s^*.
\end{align}
Combining \eqref{0fev2},  \eqref{0fev1} and \eqref{jan2021}  we infer  for all $s\geq  s^*,$ 
\begin{align}\label{2fev12}
\frac{d}{ds}\mathcal{I}_\psi(w(s),s)\  \ge&\  Q  - Q_8(s - s^*)
+\frac{C}{s^{a+b+1}}\ibint |w|^{p+1}\log^a(2+\p^2 w^2)  
 \psi^2\w\y .
%& - 4 s^{-b}\int_{\mathbb{R}^N} \psi w \nabla \psi . \nabla w \w \y. 
\end{align}
Thanks to \eqref{equiv1} and  \eqref{equiv4bis},  we have for all $s\geq  s^*,$ 
\begin{equation}\label{13janv2}
\frac{1}{s^a}\int_{\er^N}   {|  w|^{p+1}}\log^{{a}}(2+\p^2 w^2  ) \psi^2 \w \y\ge 
C\int_{\er^N}   {|  w|^{\frac{p+3}2}} \psi^2\w \y-
 C_5.
\end{equation}
Due to    Jensen inequality, \eqref{2fev12} and \eqref{13janv2} we find for all $s\geq  s^*,$ 
\begin{align}\label{bis222}
\frac{d}{ds}\mathcal{I}_\psi(w(s),s)\  \ge&\ \tilde Q  - Q_{9}(s - s^*)
+C_6\Big(\mathcal{I}_\psi(w(s),s)\Big)^{\frac{p+3}4},
%\quad \mathcal{I}_\psi(w(s^*),s^*)\ge 0,
\end{align}
where $\tilde Q=Q-C_5$.

\bigskip

It is
interesting to denote  that we 
 easily prove 
 that  the solution of the following differential inequality:
\begin{equation*}%\label{gen}
\left\{
\begin{array}{l}
h' (s)\geq 1 + C_6h^{\frac{p+3}{4}}(s),\qquad  s>s^*,\\
\\
 h(s^*) \geq 0,
\end{array}
\right.
\end{equation*}
blows up in finite time before 
$$s = s^* + \int_{0}^{+\infty}\frac{d\xi}{1 + C_6\xi^\frac{p+3}{4}} = s^* + T^*.$$
Now,  we choose $Q=
  Q_{9}T^* +C_5+ 1$ to get
%
%On the interval $[s^*, s^* + T^*]$, we have 
%$$M_1  - Q_{10}(s - s^*) \geq M_1  - Q_{10}T^*.$$
%Thus, we fix $M_1 =  Q_{10}T^* + 1$ to get
 $\tilde Q  - Q_{9}(s-s^*) \geq 1$ for all $s \in [s^*, s^* + T^*]$. \\
Therefore, $\mathcal{I}_\psi(w(s),s)$ blows up in some finite time before $s^* + T^*$. But this contradicts with the  global existence  of $w$. This implies \eqref{Eloc1} and we complete the proof of Proposition \ref{prop:upElc}.
\end{proof}
\Box

\def\cprime{$'$} \def\cprime{$'$}
\providecommand{\bysame}{\leavevmode\hbox to3em{\hrulefill}\thinspace}
\providecommand{\MR}{\relax\ifhmode\unskip\space\fi MR }
% \MRhref is called by the amsart/book/proc definition of \MR.
\providecommand{\MRhref}[2]{%
  \href{http://www.ams.org/mathscinet-getitem?mr=#1}{#2}
}
\providecommand{\href}[2]{#2}

%%%%%%%%%%%%%%%%%%%%%%%%%%%%%%%%%%%%%%%%%%%%%
%%%%%%%%%%%%%%%%%%%%%%%%%%%%%%%%%%%%%%%%%%%%%

\noindent{\bf Address}:\\
 Department of Basic Sciences, Deanship of Preparatory and Supporting Studies,
  Imam Abdulrahman Bin Faisal University
P.O. Box 1982 Dammam, Saudi Arabia.\\
% D\'epartement de Math\'ematiques,  Facult\'e des Sciences de Tunis, Universit\'e de Tunis El-Manar,  Campus Universitaire 1060, %Tunis, Tunisia.\\
\vspace{-7mm}
\begin{verbatim}
e-mail:  mahamza@iau.edu.sa
\end{verbatim}
Universit\'e Sorbonne Paris Nord, Institut Galil\'ee,
Laboratoire Analyse, G\'eom\'etrie et Applications, CNRS UMR 7539,
99 avenue J.B. Cl\'ement, 93430 Villetaneuse, France.\\
\vspace{-7mm}
\begin{verbatim}
e-mail: Hatem.Zaag@univ-paris13.fr
\end{verbatim}


\begin{thebibliography}{10}


\bibitem{CLcpde84}
T.~Cazenave and P.~L. Lions.
\newblock Solutions globales d'\'equations de la chaleur semi lin\'eaires.
\newblock {\em Comm. Partial Differential Equations}, 9(10):955--978, 1984.%ok



%
%
%
%\bibitem{Bia75}
%I.~B lynicki-Birula and J.~Mycielski.  \emph{ Wave equations with logarithmic nonlinearities},
%Bull. Acad. Pol. Sc. XXIII, 461--466, 1975.


\bibitem{BKnonl94}
J.~ Bricmont and A.~Kupiainen, \emph{Universality in blow-up for nonlinear heat equations}, Nonlinearity, \textbf{7(2)}, 539--575,
1994.



\bibitem{CMR}
C.~Collot, F.~ Merle, P.~Raphael, Dynamics near the ground state for the energy critical nonlinear heat equation
in large dimensions. Comm. Math. Phys. 352, 103-157, 2017.



\bibitem{DMW}
M.~del Pino, M.~Musso, J. Wei, Type II blow-up in the 5-dimensional energy critical heat equation. To appear
in Acta Mathematica Sinica (Special issue in honor of Carlos Kenig).





\bibitem{DVZ}
G.K.~Duong, V.T.~ Nguyen and H.~Zaag, 
\newblock {Construction of a stable blowup solution with a prescribed behavior for
 a non-scaling-invariant semilinear heat equation.}
\newblock {\em  Tunisian J. Math.}, \textbf{1(1)}, 13--45, 2019.



\bibitem{FHV}
S.~Filippas, M.A.~Herrero, J.J.L.~Vel\`azquez, Fast blow-up mechanisms for sign-changing solutions of a semilinear
parabolic equation with critical nonlinearity. R. Soc. Lond. Proc. Ser. A Math. Phys. Eng. Sci. 456, no.
2004, 2957–2982, 2000.



\bibitem{GKcpam85}
Y.~Giga and R.~V. Kohn,
\newblock Asymptotically self-similar blow-up of semilinear heat equations.
\newblock {\em Comm. Pure Appl. Math.}, 38(3):297--319, 1985.


\bibitem{GKIuMJ87}
 Y.~Giga and R.~V. Kohn, 
\newblock  Characterizing blowup using similarity variables.
\newblock {\em Indiana Univ.
Math. J.}, 36, 1–40. 1987.

\bibitem{GKCPAM89}
Y.~Giga and R.~V. Kohn, 
\newblock Nondegeneracy of blow up for semilinear heat equations.
\newblock {\em Comm.
Pure Appl. Math.}, 42, 845--884, 1989.



\bibitem{GMSiumj04}
Y.~Giga, S.~Matsui, and S.~Sasayama,
\newblock Blow up rate for semilinear heat equations with subcritical
  nonlinearity.
\newblock {\em Indiana Univ. Math. J.}, 53(2):483--514, 2004.

\bibitem{H1}
M.A.~Hamza, 
\newblock{The blow-up rate for strongly perturbed semilinear wave equations in the conformal regime without a radial assumption}.
Asymptotic Analysis,  \textbf{ 97, no. 3-4},  351--378, 2016.




\bibitem{omar1}
M.A.~Hamza and O.~Saidi, \emph{ The blow-up rate for strongly
perturbed semilinear wave equations}, 
J. Dyn. Diff.  Equ.,  \textbf{26}, 1115--1131, 2014.

\bibitem{omar2}
\bysame, \emph{
 The blow-up rate for strongly perturbed semilinear wave equations
in the conformal case},
 Math Phys Anal Geom,  \textbf{ 18(1)}, Art. 15, 2015.

\bibitem{HZjhde12}
M.A.~Hamza and H.~Zaag, \emph{A {L}yapunov functional and blow-up results for a
  class of perturbations for semilinear wave equations in the critical case},
  J. Hyperbolic Differ. Equ., \textbf{9}, 195--221, 2012.

\bibitem{HZnonl12}
\bysame, \emph{A {L}yapunov functional and blow-up results for a class of
  perturbed semilinear wave equations}, Nonlinearity, \textbf{25},
  2759--2773, 2012.



\bibitem{HZjmaa2020}
\bysame, \emph{The blow-up rate for a non-scaling invariant semilinear wave equations}, Journal of Mathematical Analysis and Applications, Volume 483
, \textbf{2}, 2020.




\bibitem{HZArxiv2020}
\bysame, \emph{The blow-up rate for a non-scaling invariant semilinear wave equations in higher dimensions.} (2020), . arXiv:2012.05374.


\bibitem{H}
J.~Harada, A type II blowup for the six dimensional energy critical heat equation, arXiv:2002.00528.




\bibitem{HV1}
M.A.~Herrero, J.J.L.~Vel\`azquez. Explosion de solutions d’equations paraboliques semilineaires supercritiques.
[Blowup of solutions of supercritical semilinear parabolic equations] C. R. Acad. Sci. Paris Ser. I Math. 319, no. 2, 141-145, 1994.


\bibitem{HV1}
 M.A.~Herrero, J.J.L.~Vel\`azquez. A blow up result for semilinear heat equations in the supercritical case. Unpublished.










\bibitem{HIHP20}
 J.~Harada, A higher speed type II blowup for the five dimensional energy critical heat equation, to appear in
Ann. Inst. H. Poincare Anal. Non Lineaire

%
%
%\bibitem{KL1cpde93}
%S.~Kichenassamy and W.~Littman, \emph{Blow-up surfaces for nonlinear wave
%  equations. {I}}, Comm. Partial Differential Equations, \textbf{18},
% 431--452, 1993.
%
%\bibitem{KL2cpde93}
%\bysame, \emph{Blow-up surfaces for nonlinear wave equations. {I}{I}}, Comm.
%  Partial Differential Equations, \textbf{18}, 1869--1899, 1993.
%% 
%
%\bibitem{KSVma14}
%R.~Killip, B.~Stovall, and M.~Vi\c{s}an, \emph{Blowup behaviour for the
%  nonlinear {K}lein--{G}ordon equation}, Math. Ann., \textbf{358},
%   289--350, 2014. 
%%
%%\bibitem{KV11}
%%R.~Killip and M.~Vi\c{s}an, \emph{Smooth solutions to the nonlinear wave
%%  equation can blow up on {C}anto


\bibitem{LSU68}
O.~A. Ladyzenskaja, V.~A. Solonnikov, and N.~N. Ural{\cprime}ceva.
\newblock {\em Linear and quasilinear equations of parabolic type}.
\newblock Translated from the Russian by S. Smith. Translations of Mathematical
  Monographs, Vol. 23. American Mathematical Society, Providence, R.I., 1968.

\bibitem{MMCPAM2004}
H.~Matano, F.~Merle.
\newblock {\em   On nonexistence of type II blowup for a supercritical nonlinear heat equation}.
\newblock  Comm. Pure
Appl. Math.,  57,  1494-1541, 2004.



\bibitem{M}
N.~Mizoguchi, Type-II blowup for a semilinear heat equation, Adv. Differential Equations 9 no. 11-12 
1279-1316, 2004.


\bibitem{V}
 V.~T.   Nguyen.
\newblock {\em  On the blow-up results for a class of strongly perturbed semilinear heat equations}.
\newblock Discrete  Continuous Dynamical Systems- A, 35 (8), 3585-3626, 2015.



\bibitem{VZ}
 V.~T.   Nguyen and H. Zaag
\newblock {\em Finite degrees of freedom for the refined blow-up profile of the semilinear heat equation.} 
\newblock Ann. Scient. Éc. Norm. Supér (4). 50:5, 1241-1282, 2017.


\bibitem{V2}
 V.~T.   Nguyen and H. Zaag
\newblock {\em Blow-up results for a   strongly perturbed semilinear heat equations: Theoretical analysis and numerical method}.
\newblock Analysis  and PDE Vol. 9, No. 1:229--257, 2016.


\bibitem{Qu99amuc}
P.~Quittner.
\newblock A priori bounds for global solutions of a semilinear parabolic
  problem.
\newblock {\em Acta Math. Univ. Comenian. (N.S.)}, 68(2):195--203, 1999.




\bibitem{sJFA2012}
R.~Schweyer.
\newblock Type II blow-up for the four dimensional energy critical semi linear
heat equation. J. Funct. Anal., 263(12):3922–3983, 2012.





\bibitem{S1}
Y.~Seki, Type II blow-up mechanisms in a semilinear heat equation with critical Joseph-Lundgren exponent, J.
Funct. Anal. 275 no. 12, 3380-3456, 2018.


\bibitem{S2}
 Y.~Seki, Type II blow-up mechanisms in a semilinear heat equation with Lepin exponent, J. Differential Equations
268 no. 3,  853-900, 2020.



\end{thebibliography}
\end{document}